\documentclass[12pt]{amsart}
\usepackage[a4paper,margin=1in]{geometry}
\usepackage{microtype}
\usepackage[cal=euler]{mathalfa}
\usepackage{tikz,tikz-cd}

\input{commands}
\usepackage[colorlinks,allcolors=blue]{hyperref}

\defcite\vlad{MR2480715}
\defcite\kac{MR1417941}
\defcite\roi{MR1935501} 
\defcite\ER{engel_smooth}
\defcite\ro{MR1700512} 
\defcite\mil{MR2967107} 
\oper{Coeff}
\def\tp#1#2{#1^{(#2)}} 
\def\prim{\mathrm{prim}}
\def\LieCon{vertex Lie\xspace}

\begin{document}
\title{DT invariants from vertex algebras}

\author{Vladimir Dotsenko} 
\address{Institut de Recherche Math\'ematique Avanc\'ee, UMR 7501, Universit\'e de Strasbourg et CNRS, 7 rue Ren\'e-Descartes, 67000 Strasbourg, France}
\email{vdotsenko@unistra.fr}
\author{Sergey Mozgovoy} 
\address{School of Mathematics, Trinity College Dublin, Dublin 2, Ireland
\newline\indent
Hamilton Mathematics Institute, Dublin 2, Ireland}
\email{mozgovoy@maths.tcd.ie}

\begin{abstract}
We obtain a new interpretation of the cohomological Hall algebra $\mathcal{H}_Q$ of a symmetric quiver $Q$ in the context of the theory of vertex algebras. Namely, we show that the graded dual of $\mathcal{H}_Q$ is naturally identified with the underlying vector space of the principal free vertex algebra associated to the Euler form of $Q$. Properties of that vertex algebra are shown to account for the key results about~$\mathcal{H}_Q$. In particular, it has a natural structure of a vertex bialgebra, leading to a new interpretation of the product of~$\mathcal{H}_Q$. Moreover, it is isomorphic to the universal enveloping vertex algebra of a certain vertex Lie algebra, which leads to a new interpretation of Donaldson--Thomas invariants of $Q$ (and, in particular, re-proves their positivity). Finally, it is possible to use that vertex algebra to give a new interpretation of CoHA modules made of cohomologies of non-commutative Hilbert schemes.
\end{abstract}

\maketitle

\section{Introduction}
Cohomological Hall algebras (abbreviated as \cohas) were introduced in \cite{kontsevich_cohomological} as a mathematical interpretation of the notion of algebra of (closed) BPS states in string theory \cite{harvey_algebras}.
In a nutshell, the definition of \coha goes through the same lines as the definition of the usual Hall algebra ~\cite{ringel_hall,lusztig_quivers},
but it uses cohomology of moduli stacks of objects instead of constructible functions on those stacks.
As a consequence, the underlying vector space of \coha is easier to describe explicitly.
In particular, for the category of quiver representations, one can obtain \cite{kontsevich_cohomological}
an explicit description of the product of CoHA using the Feigin--Odesskii shuffle product.
\medskip

On the other hand, one is also interested in representations of BPS algebras, modeled by spaces of open BPS states.
In mathematical terms this corresponds to \coha-module structures constructed on cohomology of moduli spaces of stable framed quiver representations \cite{soibelman_remarks,franzen_cohomology,franzen_chow}.
This approach is reminiscent of the construction of
Heisenberg algebra action on the cohomology (or K-theory) of Hilbert schemes on surfaces \cite{nakajima_heisenberg,grojnowski_instantons,lehn_chern,lehn_symmetric,li_vertex,schiffmann_elliptic}
or quantum group actions on equivariant cohomology (or K-theory) of Nakajima quiver varieties \cite{nakajima_quiverb}. 
Explicit description of \coha-modules, using shuffle products, was obtained in \cite{franzen_cohomology}.
\medskip 

Poincar\'e series of \cohas provide one of the possible approaches to the refined Donaldson--Thomas (DT) invariants, also called the BPS invariants \cite{kontsevich_stability}.
As a result, one can think about \cohas as a categorification of DT invariants.
Unraveling new intrinsic structures of \cohas (say, that of a Hopf algebra or that of a vertex algebra) leads to a better understanding of properties of DT invariants.
\medskip

Vertex algebras were introduced in \cite{MR843307} in order to formalise the machinery behind some remarkable constructions in representation theory of infinite-dimensional Lie algebras \cite{MR595581,MR747596}. 
In modern language those constructions are instances of lattice vertex algebras.
More general vertex algebras
provide a mathematical apparatus to work with operator product expansions (OPEs) in quantum field theory \cite{MR260310,MR295712}.
For us, a special role will be played by 
the principal subalgebras 
of lattice vertex algebras
\cite{MR1275728,MR2967107}
which give explicit realisations of the so called free vertex algebras \cite{MR1700512,MR1935501}. 
\medskip


Operator product expansions $A(z)B(w)=\sum_{n\in\bZ}(z-w)^n O_n(w)$ have a regular part (corresponding to non-negative $n$) and a singular part.
If one is only interested in the singular parts of OPEs, this leads to the notion of \LieCon algebras, also known as Lie conformal algebras \cite{MR1417941,MR1670692}.
By contrast to vertex algebras these structures only use $n$-products for $n\ge0$.
The relationship between \LieCon algebras and vertex algebras is closely resembling that between Lie algebras and associative algebras.
Notably, many important vertex algebras,
and in particular free vertex algebras mentioned earlier,
arise as universal envelopes of \LieCon algebras. 
These universal enveloping vertex algebras are cocommutative vertex bialgebras (meaning vertex algebras with a compatible coalgebra structure)
and are automatically universal enveloping algebras of Lie algebras.
Thus they possess a Hopf algebra structure, and this is where their relationship to \cohas manifests itself.
\medskip

Exploring the relationship between \cohas and vertex algebras
is the main topic of this paper.
This relationship is not entirely new and there seem to be some physical reasons for it~
\cite{rapcak_cohomological}.
For example,
in \cite[Theorem B]{schiffmann_cherednik} (see also \cite{rapcak_cohomological}) 
a relationship between 
an equivariant spherical \coha of a $3$-loop quiver 
and the current algebra \cite[\S3.11]{arakawa_representation} of some vertex algebra was established.
On the other hand, in \cite{joyce_ringel} 
there was given a very general construction of vertex algebra structures on homology groups of various moduli stacks, which leads to a vertex algebra structure on the duals of cohomological Hall algebras. 
\medskip

In this paper, we traverse the path between \cohas and vertex algebras in a different direction. Namely, given a symmetric quiver $Q$, we show that, for the principal free vertex algebra corresponding to the Euler form of $Q$, the shuffle product formula of CoHA emerges ``for free'', since these free vertex algebras are universal envelopes of appropriate \LieCon algebras. 
\medskip

In \cite[\S6.2]{kontsevich_cohomological} 
it was asked if \coha can be identified with a Hopf algebra of the form $U(\fg)$, where $\fg$ is a graded Lie algebra. In physics terms, this Lie algebra should correspond to the space of single-particle BPS states \cite{harvey_algebras,rapcak_cohomological}. Of course, in view of Milnor--Moore theorem, one just needs to find a suitable cocommutative bialgebra structure on \coha. 
A bialgebra structure for usual Hall algebras is guaranteed by Green's theorem \cite{green_hall} and appropriate twists from ~\cite{xiao_drinfeld}. 
A cohomological incarnation of Green's theorem
was proved in \cite{davison_critical} 
for some localization of \cohas for quivers with potential
(see also \cite{yang_cohomological} for a purely algebraic approach to a similar problem).
It is unclear, however, what are the implications of this result for a bialgebra structure on the \coha itself.
For a different cocommutative coproduct for usual Hall algebras see \cite{ringel_lie,joyce_configurationsa,bridgeland_stabilitya}.

\medskip

For a symmetric quiver $Q$, the algebra $\cH_Q$ is supercommutative (once certain sign twists are implemented), and thus the required Lie algebra $\fg$ has to be abelian.
In this case, it was conjectured in \cite[\S2.6]{kontsevich_cohomological} 
and proved in \cite{MR2956038}
that there exists an abelian Lie algebra $\fg$ such that $\cH_Q\iso U(\fg)$ and \fg is of the form $\fg^\prim[x]$ with $x$ of cohomological degree $2$.
Further generalizations for quivers with potentials were obtained in \cite{davison_cohomological}. 
\medskip

In our approach to the above question we change the setting and we argue that
the graded dual $\cH_Q\dual$, which is a 
(super) cocommutative coalgebra, 
has a vertex bialgebra structure and 
can be represented as the universal enveloping algebra of a canonical Lie algebra
(non-abelian in general).
More precisely, 
we establish the following result.

\begin{theorem}[{Theorem \ref{th:DT}}]
Let $Q$ be a symmetric quiver, $L=\bZ^{Q_0}$ 
and $\hi$ be the Euler form of~$Q$.
Then the coalgebra $\cH_Q\dual$ has a canonical structure of a cocommutative connected vertex bialgebra.
The space of primitive elements 
$$C=P(\cH_Q\dual)\in\Vect^{L\xx\bZ}$$
is a vertex Lie algebra (also having a structure of a Lie algebra)
such that $\cH_Q\dual$ is isomorphic to the universal enveloping vertex algebra of $C$ (as vertex bialgebras).
The canonical derivation $\dd$ on $C$ has $L$-degree zero and cohomological degree $-2$,
and $C$ is a free $\bQ[\dd]$-module 
such that
the space of generators 
$$C/\dd C=\bop_{\bd\in L}W_\bd=\bop_{\bd\in L,k\in\bZ}W_\bd^k$$ 
has finite-dimensional components $W_\bd$
and $k\equiv\hi(\bd,\bd)\pmod 2$ whenever $W_\bd^k\ne0$. 
\end{theorem}


The coalgebra $\cH_Q\dual$ is isomorphic to the universal enveloping algebra $U(C)$, where we interpret $C$ as a Lie algebra.
Consequently, the refined Donaldson--Thomas invariants of ~$Q$ are the characters of the components of $C/\dd C$ and are contained in $\bN[q^{\pm\oh}]$, 
hence we obtain a new proof of positivity of DT invariants for symmetric quivers, originally proved in \cite{MR2956038}.
Note that the quotient $C/\dd C$ has a canonical Lie algebra structure, see \S\ref{lie conformal}.
The Lie algebra $C/\dd C$ is different from the BPS Lie algebra introduced in \cite{davison_cohomological} (although the characters of both algebras compute DT invariants), which is a Lie subalgebra of $\cH_Q$ and has, in particular, the trivial bracket.
We conjecture that our results can be generalized
to \cohas associated to (symmetric) quivers with potential.
\medskip

Our approach also allows one to re-interpret the \COHA-modules $\cM_\bw$ arising from moduli spaces of stable framed representations (also known as non-commutative Hilbert schemes). In particular, we construct combinatorial spanning sets for duals of those modules; this result is a substantial generalisation (and a conceptual interpretation) of the main theorem of \cite{MR2480715}.
\medskip

Let us mention that 
all \LieCon algebras that we consider naturally lead
to Lie algebras with quadratic relations.
Moreover,  principal free vertex algebras are modules with quadratic relations over these Lie algebras,
so one may wish to study them in the context of the Koszul duality theory. This approach is taken in the paper~\cite{DoFeRe}.  

\subsection*{Structure of the paper. }

In \S\ref{sec:conventions}, we summarize various conventions used throughout the paper. In particular, for a free abelian group $L$ equipped with an integer-valued symmetric bilinear form, we introduce in \S\ref{sec:graded} a special braiding for the monoidal category of $L$-graded vector spaces; in this way we completely avoid any non-canonical choices of sign twists needed both to ensure super-commutativity of CoHA \cite[\S2.6]{kontsevich_cohomological} and mutual locality of vertex operators in lattice vertex algebras \kac[\S5.4].  
In \S\ref{sec coha}, we recall the necessary definitions and results of the theory of cohomological Hall algebras and their modules.
In \S\ref{sec:VA}, we recall the necessary definitions and results of the theory of vertex algebras and conformal algebras, adapted to the symmetric monoidal category of $L$-graded vector spaces with the braiding arising from a symmetric bilinear form.
In \S\ref{coha and free}, we prove the main result of this paper: a new interpretation of CoHA $\cH_Q$ as the graded dual of the principal free vertex algebra $\cP_Q$. To that end, we obtain a natural isomorphism of graded vector spaces $\cP_Q\dual\iso\cH_Q$ (Proposition \ref{prop:sympoly}), and then show that the canonical coalgebra structure on $\cP_Q$ leads to the shuffle product of CoHA (Theorem \ref{th:VOA-CoHA}). As a consequence, we show that the dual of CoHA is identified with the universal envelope of a Lie algebra with a freely acting derivation (Theorem \ref{th:DT}), leading to a new proof of positivity of DT invariants (Corollary \ref{cor:DT}). 
In \S\ref{sec:iso-modules}, we give two new descriptions of the dual space of the CoHA-module $\cM_\bw$: as the kernel of the appropriate reduced coaction map (Theorem \ref{th:modkernel}) and via an explicit combinatorial spanning set inside $\P_Q$ (Theorem \ref{module spanning}). We also use those results to establish a surprising symmetry result for the positive and the negative halves of the coefficient Lie algebra, which may be interpreted as strong evidence for the Koszulness conjecture of~\cite{DoFeRe}.

\subsection*{Acknowledgements} The first author is grateful to Boris Feigin for introducing him to the beauty of principal subalgebras of lattice vertex algebras, and to Michael Finkelberg and Valery Lunts whose questions about a decade ago made him suspect that his work in~\cite{MR2480715} may be related to non-commutative Hilbert schemes. The second author is grateful to Ben Davison, Hans Franzen, Boris Pioline, Markus Reineke, Olivier Schiffmann and Yan Soibelman for useful discussions.
Both authors would like to thank
Arkadij Bojko, Evgeny Feigin and  Dominic Joyce for helpful feedback on a draft version of the paper.

This work has benefited from support provided by University of Strasbourg Institute for Advanced Study (USIAS) for the Fellowship USIAS-2021-061, within the French national program ``Investment for the future'' (IdEx-Unistra), by the French national research agency project ANR-20-CE40-0016, and by Institut Universitaire de France.
\section{Conventions}\label{sec:conventions}

Unless specified otherwise, all vector spaces and (co)chain complexes in this article are defined over the ground field of rational numbers.
We use cohomological degrees, and view homologically graded complexes as cohomologically graded ones: 
for a chain complex $C_*$, we consider the cochain complex $C^*$ with $C^{n}=C_{-n}$
(note that $C[k]^n=C^{n+k}=C_{-n-k}=C[k]_{-n}$).
Throughout the paper, $L$ denotes a free abelian group equipped with an integer-valued symmetric bilinear form $(\cdot,\cdot)$. 

Our work brings together two different worlds: that of cohomological Hall algebras and that of vertex operator algebras. The former operates within the derived category $D(\Vect)$ which we identify with $\Vect^\bZ$, the latter uses half-integer conformal weights and thus operates within the category $\Vect^{\oh\bZ}$. Both traditions are well established, so we decided to not break either of them, but rather make the necessary effort to carefully translate results from one language to another.

\subsection{Borel--Moore homology}
For an introduction to Borel--Moore homology see \eg \cite{chriss_representation} and for an introduction to equivariant Borel--Moore homology see \eg \cite{edidin_equivariant}.
Given an algebraic variety $X$ over $\bC$, we define its Borel--Moore homology $H^\BM(X)$ as
$$H^\BM(X)^{-n}=H^\BM_n(X)=H^n_c(X)\dual,$$
where $H^n_c(X)$ denotes the cohomology with compact support and coefficients in $\bQ$.
If $X$ is smooth and has dimension $d_X$, then the Poincar\'e duality implies that
$$H^\BM(X)\iso H^*(X)[2d_X].$$
Given an algebraic group $G$ acting on $X$, we define the equivariant Borel--Moore homology $H^\BM_G(X)=H^\BM([X/G])$ as
$$H^\BM_G(X)^{-n}=H^\BM_{G,n}(X)=H^n_{c,G}(X)\dual,$$
where $H^n_{c,G}(X)$ denotes $G$-equivariant cohomology with compact support and coefficients in $\bQ$.
If $X$ is smooth and $X,G$ have dimensions $d_X,d_G$ respectively, then
$$H^\BM_G(X)\iso H^*_G(X)[2d_X-2d_G].$$

\subsection{Graded vector spaces}\label{sec:graded}

Let us consider the category $\Vect^L$ of $L$-graded vector spaces $V=\bop_{\al\in L}V_\al$, with morphisms of degree zero.
It has a structure of a closed monoidal category with tensor products and internal \Hom-objects defined by
$$V\ts W=\bop_{\al\in L}(V\ts W)_\al,\qquad
(V\ts W)_\al=\coprod_{\beta\in L}V_{\beta}\ts W_{\al-\beta},$$
$$\lHom(V,W)=\bop_{\al\in L}\Hom_\al(V,W),\qquad
\Hom_\al(V,W)=\prod_{\beta\in L}\Hom(V_\beta,W_{\al+\beta}).$$
We equip $\Vect^L$ with a symmetric monoidal category structure, where the braiding morphism is defined using the symmetric bilinear form $(\cdot,\cdot)$ on $L$:
\begin{equation}
\si\colon V\ts W\to W\ts V,\qquad a\ts b\mto(-1)^{(\al,\beta)}b\ts a,\qquad
a\in V_\al,\, b\in W_\beta.
\end{equation}
In what follows, we shall often use $\cC$ to denote $\Vect^L$ equipped with the thus defined symmetric closed monoidal category structure.

We define associative algebras in $\cC$ as monoid objects in this category; in particular,for each $V\in\cC$, the object $\lEnd(V)=\lHom(V,V)\in\cC$ is an associative algebra. Using the braiding~$\si$, one may also define commutative algebras and Lie algebras, and their modules. (Alternatively, one may note that the category $\cC$ contains the category $\Vect$ as a full symmetric monoidal subcategory of objects of degree zero, and so one may consider objects in $\cC$ which are algebras over the classical operads $\mathsf{Ass}$, $\mathsf{Com}$, and $\mathsf{Lie}$ in $\Vect$.) In particular, as in the case of $\Vect$, the free associative algebra generated by an object $X$ of $\cC$ is the tensor algebra $\mathsf{T}(X)=\coprod_{n\ge0} X^{\otimes n}$, and the free commutative algebra generated by an object $X$ of $\cC$ is the symmetric algebra $S(X)=\coprod_{n\ge0} (X^{\otimes n})_{\Si_n}$. 

Given an associative algebra $A\in\cC$, we may equip it with the bracket
$$[-,-]\colon A\ts A\to A,\qquad [a,b]=\mu(a\ts b)-\mu\si(a\ts b) ;$$
this defines a functor from the category of associative algebras in $\cC$ to the category of Lie algebras in $\cC$. This functor has a left adjoint functor, the functor of the universal enveloping algebra $U(\fg)$ of a Lie algebra $\fg$. 
We shall use two different versions of the Poincar\'e--Birkhoff--Witt theorem for universal enveloping algebras. First of them asserts that if $\fg$ is a Lie algebra and $\fh\subset\fg$ is a Lie subalgebra, the universal enveloping algebra $U(\fg)$ is a free $U(\fh)$-module, and the vector space of generators of this module is isomorphic to $S(\fg/\fh)$; moreover, if there exists a Lie subalgebra $\fh'$ such that $\fg=\fh\oplus \fh'$, we may take $U(\fh')\subset U(\fg)$ as the space of generators.
The second version uses the fact that $U(\fg)$ has a canonical coproduct $\Delta$ for which elements of $\fg$ are primitive, and this coproduct makes $U(\fg)$ a cocommutative coassociative coalgebra. The theorem asserts that as a coalgebra, $U(\fg)$ is isomorphic to $S^c(\fg)$, the cofree cocommutative coassociative conilpotent coalgebra generated by $\fg$ (note that since we work over $\bQ$, the underlying object of $S^c(X)$ is the same as that of $S(X)$ for all $X$ in $\cC$). 
To prove the Poincar\'e--Birkhoff--Witt theorem in $\cC$, one may use the methods of \cite{MR506890} or \cite{dotsenko2018endofunctors} for the first version and the methods of \cite[Appendix B]{MR258031} or \cite{MR2504663} for the second version.    

We note that for an abelian group $L$ equipped with a homomorphism $p\colon L\to\bZ_2$ (for example, the parity $\bZ\to\bZ_2$),
one normally thinks of the category of $L$-graded vector spaces with the braiding morphism (the \idef{Koszul sign rule})
\begin{equation}\label{koszul}
\si\colon V\ts W\to W\ts V,\qquad a\ts b\mto(-1)^{p(\al)p(\beta)}b\ts a,\qquad
a\in V_\al,\, b\in W_\beta.
\end{equation}
and refers to the corresponding commutative (or Lie) algebras as super-commutative algebras (or Lie superalgebras). In this paper we mostly encounter the more general setting discussed above, and suppress the qualifier ``super''.

\subsection{Laurent series}

Let $V$ be a vector space. The vector space of doubly infinite Laurent series with coefficients in $V$ is denoted by $V\pser{z^{\pm1}}$, and its subspace of formal Laurent series (that is, series for which only finitely many negative components are non-zero) is denoted by $V\lser{z}\subset V\pser{z^{\pm1}}$. 
For a series $v=\sum_{n\in\bZ} v_nz^n \in V\pser{z^{\pm1}}$, 
we denote the coefficient $v_{-1}$ by $\Res_z v$. 

Formal Laurent series with coefficients in an algebra do themselves form an algebra. It is important to note that doubly infinite Laurent series in several variables contain several different subspaces of formal Laurent series in those variables: for instance, the subspaces $\bQ\lser{z}\lser{w}$ and $\mathbb{Q}\lser{w}\lser{z}$ of $\bQ\pser{z^{\pm1},w^{\pm1}}$ are different. A lot of formulas in the theory of vertex algebras use the fact that rational functions like $\frac1{z-w}$ can be expanded as elements of both of these rings. To avoid unnecessarily heavy formulas, we shall frequently use the \emph{binomial expansion convention} for such expansion, see \eg \cite[\S2.2]{lepowsky_introduction}:
for $n\in\bZ$, we define the formal Laurent series
$(z+w)^n$ by the formula
$$(z+w)^n
=\sum_{k\ge 0}\binom{n}{k}z^{n-k}w^k.$$
In plain words, we expand powers of binomials as power series in the second summand. Thus, for example, 
$(z-w)^n=(-1)^n(-z+w)^n$
for all $n\in\bZ$, but it is equal to $(-1)^n(w-z)^n$ only for $n\ge0$. 
An important doubly infinite Laurent series 
in $\bQ\pser{z^{\pm1}}$
is the ``delta function''
 \[
\delta(z)=\sum_{n\in\mathbb{Z}}z^n.
 \]
Doubly infinite Laurent series cannot be multiplied, but in the instances where we use the above series, we shall use the following version which can be multiplied by any formal Laurent series in $w$: 
$$\de\rbr{\frac{z-w}u}=\sum_{n\in\bZ}\sum_{k\ge0}(-1)^k\binom nk u^{-n}z^{n-k}w^k.$$
A useful formula involving this expression is $\Res_z \de\rbr{\frac{z-w}u}=u$.  

\subsection{Characters and Poincar\'e series}
\label{sec:char}
Given an object $M\in\Vect^\bZ$ with finite-dimen\-sional components, we define its \idef{character} $\ch(M)$ by the formula
\begin{equation}
\ch(M)=\sum_{k\in\bZ}\dim M^kq^{-\oh k}.
\end{equation}
The number $-\oh k$ will be called the \idef{weight} of the component $M^k$ (it may be convenient to view $-k$ as a homological degree).
Note that 
$$\ch(M[n])=q^{\oh n}\ch(M).$$
In particular, for an algebraic variety $X$, we have
$$\ch(H^\BM(X))
=\sum_{k\in\bZ} \dim H^\BM(X)^{-k}q^{\oh k}
=\sum_{k\in\bZ} \dim H^k_c(X)q^{\oh k}
=P_c(X,q),$$
the \idef{Poincar\'e polynomial} (with compact support) of $X$.

More generally, given an object $M\in(\Vect^{L})^\bZ\iso\Vect^{L\xx\bZ}$ with finite-dimensional components $M_\bd^k$, for $\bd\in L$ and $k\in\bZ$, we define its \idef{Poincar\'e series} $Z(M,x,q)$ by the formula
\begin{equation}\label{part function}
Z(M,x,q)=\sum_{\bd\in L}(-1)^{(\bd,\bd)}\ch(M_\bd)x^\bd
=\sum_{\bd\in L}\sum_{k\in\bZ}(-1)^{(\bd,\bd)}
\dim M_\bd^k \cdot q^{-\oh k}x^\bd.
\end{equation}

\subsection{Plethystic exponential}
For more information on \la-rings and plethystic exponentials see \eg
\cite{getzler_mixed,mozgovoy_computational}.
Consider the ring
$$R=\bQ\lser{q^\oh}\pser{x_i\col i\in I},\qquad I=\set{1,\dots,r},$$
and its maximal ideal $\fm$.
We define \idef{plethystic exponential} to be the group isomorphism
$$\Exp\colon (\fm,+)\to(1+\fm,*)$$
defined on monomials by
$$\Exp(q^k x^\bd)=\sum_{n\ge0}q^{nk}x^{n\bd},\qquad k\in\oh\bZ,\,\bd\in\bN^I\ms\set0.$$
Let us assume now that $L=\bZ^I$ is equipped with a symmetric bilinear form $(\cdot,\cdot)$ and let $\Vect^{L\xx\bZ}$ be the corresponding symmetric monoidal category (with the signs in the braiding depending just on $L$-degrees).
We consider a subcategory $\cA\sbs\Vect^{L\xx\bZ}$ consisting of objects 
$$M=\bop_{\bd\in\bN^I}M_\bd
=\bop_{\bd\in\bN^I}\bop_{k\in\bZ}M_\bd^k$$
such that $M_\bd^k$ are finite-dimensional and $M_\bd^k=0$ for $k\gg0$.
Then the Poincar\'e series \eqref{part function}
induces a ring homomorphism
$$Z\colon K_0(\cA)\to R.$$
This homomorphism is actually a \la-ring homomorphism (with the \la-ring structure on $K_0(\cA)$ induced by the symmetric monoidal category structure on \cA that we defined earlier, see \eg~\cite{getzler_mixed,heinloth_note}).
We will formulate this fact in the following way.

\begin{theorem}
For any graded space $M\in\cA$ with $M_0=0$, we have
$$Z(S(M))=\Exp(Z(M)).$$
\end{theorem}
\begin{proof}
It is enough to prove the statement for a one-dimensional space $M=M_\bd^k$.
If $(\bd,\bd)$ is odd, then $Z(M)=-q^{-\oh k}x^\bd$. On the other hand, $S(M)=\bQ\oplus M$, hence
$$Z(S(M))=1-q^{-\oh k}x^\bd
=\Exp\rbr{-q^{-\oh k}x^\bd}
=\Exp(Z(M)).
$$
If $(\bd,\bd)$ is even, then $Z(M)=q^{-\oh k}x^\bd$.
On the other hand,
$S(M)
=\coprod_{n\ge0}(M^{\ts n})_{\Si_n}
=\coprod_{n\ge0}M^{\ts n}$, hence
$$Z(S(M))=\sum_{n\ge0}q^{-n\oh k}x^{n\bd}
=\Exp\rbr{q^{-\oh k}x^\bd}
=\Exp(Z(M)).
$$
\end{proof}

\section{\COHA and \COHA-modules}\label{sec coha}

Let $Q$ be a symmetric quiver with the set of vertices $I$.
In this section, we consider the abelian group $L=\bZ^{I}$ equipped with the Euler form 
 \[
\hi(\bd,\be)
=\sum_{i\in Q_0}\bd_i\be_i-\sum_{(a:i\to j)\in Q_1} \bd_i\be_j,
 \]
which is a symmetric bilinear form.

\subsection{Definition of \COHA}
For more details on the results of this section see \cite{kontsevich_cohomological}.
For any $\bd\in\bN^{I}$, we define the space of representations
$$R_\bd=R(Q,\bd)=\bop_{a\colon i\to j}\Hom(\bC^{\bd_i},\bC^{\bd_j})$$
which has the standard action of $G_\bd=\prod_{i\in I}\GL_{\bd_i}(\bC)$.
We define the \idef{cohomological Hall algebra} (\COHA) with the underlying $L$-graded object in $D(\Vect)\iso\Vect^\bZ$ equal
\begin{equation}
\cH_Q=\bop_{\bd\in\bN^I}\cH_{Q,\bd},
\end{equation}
where
\begin{equation}
\cH_{Q,\bd}
=H^\BM_{G_\bd}(R_\bd)[\hi(\bd,\bd)]
\iso H^*_{G_\bd}(R_\bd)[-\hi(\bd,\bd)]
\end{equation}
is considered as an object of $D(\Vect)$
(note that the stack $[R_\bd/G_\bd]$ has dimension $-\hi(\bd,\bd)$).
Multiplication in this algebra is constructed as follows.
Given $\bd,\be\in\bN^I$, consider 
$$V_\bd=\bop_{i\in I}\bC^{\bd_i}\sbs
V_{\bd+\be}=\bop_{i\in I}\bC^{\bd_i+\be_i}$$
and let $R_{\bd,\be}\sbs R_{\bd+\be}$ be the subspace of representations that preserve $V_\bd$.
It is equipped with an action of the parabolic subgroup $G_{\bd,\be}\sbs G_{\bd+\be}$ consisting of maps that preserve $V_\bd$.
We have morphisms of stacks
$$[R_\bd/G_\bd]\xx[R_\be/G_\be]
\xlto{q}[R_{\bd,\be}/G_{\bd,\be}]
\xto{p}[R_{\bd+\be}/G_{\bd+\be}],$$
where $q$ has dimension $-\hi(\be,\bd)$.
These maps induce morphisms in $D(\Vect)$
$$H^\BM_{G_\bd}(R_\bd)\ts H^\BM_{G_\be}(R_\be)
\xto{q^*}
H^\BM_{G_{\bd,\be}}(R_{\bd,\be})[2\hi(\be,\bd)]
\xto{p_*}
H^\BM_{G_{\bd+\be}}(R_{\bd+\be})[2\hi(\be,\bd)].
$$
Taking the composition, we obtain the multiplication map
$$\cH_{Q,\bd}\ts\cH_{Q,\be}\to\cH_{Q,\bd+\be}$$
in $D(\Vect)$.
It was proved in \cite{kontsevich_cohomological} that this multiplication is associative.

\subsection{Shuffle algebra description}
\label{sec:sguffle}
For any $n\ge0$, define the graded algebra
$$
\La_n=\bQ[x_1,\dots,x_n]^{\Si_n},$$
where $x_i$ has degree $2$ and $\Si_n$ is the symmetric group on $n$ elements.
Similarly, for any $\bd\in\bN^{I}$, define
$$\La_\bd=\bts_{i\in I}\La_{\bd_i}
=\bQ[x_{i,k}\col 
i\in I,\, 1\le k\le \bd_i]^{\Si_\bd},\qquad
\Si_\bd=\prod_{i\in I}\Si_{\bd_i}.$$
Then
$$
H^*_{\GL_n}(\pt)\iso \La_n,\qquad
H^*_{G_\bd}(\pt)\iso \La_\bd.$$
This implies that
\begin{equation}\label{coha comp1}
\cH_{Q,\bd}
=H^*_{G_\bd}(R_\bd)[-\hi(\bd,\bd)]
=\La_\bd[-\hi(\bd,\bd)].
\end{equation}
It was proved in \cite{kontsevich_cohomological} that 
the product $\cH_{Q,\bd}\ts\cH_{Q,\be}\to\cH_{Q,\bd+\be}$ is given by the \idef{shuffle product}
\begin{equation}\label{eq:shuffleprod}
f*g=\sum_{\si\in\Sh(\bd,\be)}\si(fgK)
\end{equation}
where the sum runs over all $(\bd,\be)$-shuffles, meaning $\si\in\Si_{\bd+\be}$ satisfying $$\si_i(1)<\dots<\si_i(\bd_i),\qquad
\si_i(\bd_i+1)<\dots<\si_i(\bd_i+\be_i)
\qquad \forall i\in I,$$
and the kernel $K$ is a function in the localization of $\La_\bd\ts\La_\be$ defined by
\begin{equation}
K(x,y)
=\frac
{\prod_{a:i\to j}\prod_{k=1}^{\bd_i}\prod_{\ell=1}^{\be_j}
(y_{j,\ell}-x_{i,k})}
{\prod_{i\in I}\prod_{k=1}^{\bd_i}\prod_{\ell=1}^{\be_i}
(y_{i,\ell}-x_{i,k})}
=\prod_{i,j\in I}
\prod_{k=1}^{\bd_i}\prod_{\ell=1}^{\be_j}
(y_{j,\ell}-x_{i,k})^{-\hi(i,j)}.
\end{equation}

The above formula for the shuffle product implies that
\begin{equation}
f*g=(-1)^{\hi(\bd,\be)}g*f,\qquad
f\in\cH_{Q,\bd},\,g\in\cH_{Q,\be}.
\end{equation}
This formula implies that $\cH_Q$ is a commutative algebra in the symmetric monoidal category $\Vect^{L\xx\bZ}\iso(\Vect^\bZ)^L$
with the braiding arising from the Euler form $\hi$.

\begin{remark}\label{rm:super}
It was observed in \cite[\S2.6]{kontsevich_cohomological} that it is possible to modify multiplication in $\cH_Q$ (non-canonically) to make it super-commutative.
For that, one defines the parity map $p\colon L\to \bZ_2$, $\bd\mto\hi(\bd,\bd)\pmod 2$ and chooses a group homomorphism $\eps\colon L\xx L\to\mu_2=\set{\pm1}$ such that 
\begin{equation}
\eps(\bd,\be)
=(-1)^{p(\bd)p(\be)+\hi(\bd,\be)}
\eps(\be,\bd),\qquad \bd,\be\in L.
\end{equation}
Then the product on $\cH_Q$ defined by 
\begin{equation}\label{eq:modif}
f\star g=\eps(\bd,\be)f*g,\qquad
f\in\cH_{Q,\bd},\,g\in\cH_{Q,\be},
\end{equation}
is super-commutative: $f\star g=(-1)^{p(\bd)p(\be)}g\star f$ for all $f\in\cH_{Q,\bd}$, $g\in\cH_{Q,\be}$. We shall avoid non-canonical choices and not use this sign twist.
Instead we shall interpret $\cH_Q$ as a commutative algebra in $\Vect^{L\xx\bZ}$ or $\Vect^L$ with the symmetric monoidal category structure arising from $\hi$.
\end{remark}

\subsection{Modules over \COHA}\label{sec:modules}
For more details about modules over \cohas see \eg \cite{soibelman_remarks,franzen_cohomology}.
Let $\bw\in\bN^{I}$ be a vector, called a \idef{framing vector}.
We may consider a new (framed) quiver $Q^{\bw}$ by adding a new vertex $\infty$ and $\bw_i$ arrows $\infty\to i$ for all $i\in I$.
For any $\bd\in\bN^{I}$, let $\bar\bd=(\bd,1)\in\bN^{Q_0^\bw}$
and
$$R_{\bd,\bw}^\f
=R(Q^\bw,\bar\bd)
=R_\bd\oplus F_{\bd,\bw}
=R_\bd\oplus\bop_{i\in I}\Hom(\bC^{\bw_i},\bC^{\bd_i}).$$
Let $R^{\f,\st}_{\bd,\bw}\sbs R^\f_{\bd,\bw}$ be the open subset of stable representations, consisting of representations $M$ generated by $M_\infty$.
Then $G_\bd$ acts freely on $R^{\f,\st}_{\bd,\bw}$ and we consider the moduli space
$$\Hilb_{\bd,\bw}=R^{\f,\st}_{\bd,\bw}/G_\bd,$$
called the \idef{non-commutative Hilbert scheme}. Using this moduli space, one can define a module over \COHA, denoted by $\cM_\bw$. Its underlying $L$-graded object in $D(\Vect)\iso\Vect^\bZ$ is
$$\cM_\bw
=\bop_{\bd\in\bN^I}\cM_{\bw,\bd},\qquad
\cM_{\bw,\bd}=H^\BM(\Hilb_{\bd,\bw})[\hi(\bd,\bd)-2\bw\cdot\bd].$$
The \COHA-action is defined as follows.
Let $R^\f_{\bd,\be,\bw}\sbs R^\f_{\bd+\be,\bw}$ be the subspace of framed representations that preserve $V_\bd=\bop_{i\in I}\bC^{\bd_i}$ and let $R^{\f,\st}_{\bd,\be,\bw}\sbs R^\f_{\bd,\be,\bw}$ be the open subset of stable representations.
It is equipped with a free action of the group $G_{\bd,\be}$ and we define 
$\Hilb_{\bd,\be,\bw}=R^{\f,\st}_{\bd,\be,\bw}/G_{\bd,\be}$.
We have morphisms of stacks and algebraic varieties
$$[R_\bd/G_\bd]\xx\Hilb_{\be,\bw}
\xlto{q}\Hilb_{\bd,\be,\bw}
\xto{p}\Hilb_{\bd+\be,\bw},$$
where $q$ has dimension $n=-\hi(\be,\bd)+\bw\cdot\bd$.
These maps induce morphisms in $D(\Vect)$
$$H^\BM_{G_\bd}(R_\bd)\ts 
H^\BM(\Hilb_{\be,\bw})
\xto{q^*}
H^\BM_{}(\Hilb_{\bd,\be,\bw})[-2n]
\xto{p_*}
H^\BM_{}(\Hilb_{\bd+\be,\bw})[-2n].
$$
Taking the composition, we obtain the action map
$$\cH_{Q,\bd}\ts\cM_{\bw,\be}\to\cM_{\bw,\bd+\be}$$
in $D(\Vect).$
The compatibility with the product of \COHA is proved in the same way as the associativity of that product.

The module $\cM_\bw$ can be also described using shuffle algebras \cite{franzen_cohomology,franzen_chow}.
Consider the forgetful map $j\colon R^{\f,\st}_{\bd,\bw}\to R_\bd$ and the corresponding map of stacks $$j\colon \Hilb_{\bd,\bw}\to[R_\bd/G_\bd]$$
having dimension $\bw\cdot\bd$.
It induces a map
$$j^*\colon H^\BM_{G_d}(R_\bd)\to H^\BM(\Hilb_{\bd,\bw})[-2\bw\cdot\bd]$$
which induces a map of $L$-graded objects in $D(\Vect)$
\begin{equation}
j^*\colon \cH_Q=\bop_{\bd}\cH_{Q,\bd}\to\cM_\bw=\bop_\bd\cM_{\bw,\bd}.
\end{equation}
It was proved in \cite[Theorem 5.2.1]{franzen_chow} that this map is an epimorphism of $\cH_Q$-modules and that its kernel is equal to
\begin{equation}\label{coha module1}
\ker (j^*)=
\sum_{\bd'\ge0,\bd>0}\cH_{Q,\bd'}* e_\bd^\bw\cH_{Q,\bd}
=\rbr{e_\bd^\bw\cH_{Q,\bd}\col \bd>0},
\end{equation}
where the last expression means the ideal \wrt the product in $\cH_Q$,
$$e_{\bd}^{\bw}=\prod_{i\in I}\prod_{k=1}^{\bd_i}x_{i,k}^{\bw_i}\in\La_\bd$$
is the product of appropriate powers of elementary symmetric functions and $e_\bd^\bw\cH_{Q,\bd}$ means the product in $\La_\bd$ (corresponding to the cup product in cohomology).

\subsection{Characters of \COHA and \COHA modules}\label{sec:CharCoHA}

Let us collect some formulas for the Poincar\'e series \eqref{part function} of our objects of interest. 
We begin with the Poincar\'e series of \COHA. 

\begin{proposition}\label{char coha}
We have
$$Z(\cH_Q,x,q)
=A_Q(x,q)
:=\sum_{\bd\in\bN^I}
\frac{(-q^\oh)^{-\hi(\bd,\bd)}}{(q\inv)_\bd}x^\bd,
$$
where $(q)_\bd=\prod_{i\in I} (q)_{\bd_i}$
and $(q)_n=\prod_{k=1}^n(1-q^k)$.
\end{proposition}
\begin{proof}
Consider $\La_n=\bQ[x_1,\dots,x_n]^{\Si_n}\iso\bQ[e_1,\dots,e_n]$, where $\deg x_i=2$ and $\deg e_i=2i$.
Then
$$\ch(\La_n)=\prod_{k=1}^n\frac1{1-q^{-k}}=\frac1{(q\inv)_n}.$$
This implies that 
$\ch(\La_\bd)=\frac1{(q\inv)_\bd}$
and therefore $\cH_{Q,\bd}=\La_\bd[-\hi(\bd,\bd)]$ has the character
$$\ch(\cH_{Q,\bd})
=\sum_k\dim\La_\bd^{k-\hi(\bd,\bd)}q^{-\oh k}
=q^{-\oh\hi(\bd,\bd)}\ch(\La_\bd)
=\frac{q^{-\oh\hi(\bd,\bd)}}{(q\inv)_\bd}.$$
\end{proof}

\begin{remark}
We also have
$$A_Q(x,q)
=\sum_\bd
\frac{(-q^\oh)^{-\hi(\bd,\bd)}}{(q\inv)_\bd}x^\bd
=\sum_\bd (-q^\oh)^{\hi(\bd,\bd)}
\frac{P_c(R_\bd,q)}{P_c(G_\bd,q)}x^\bd
$$
where we used the fact that $P_c(\bC^n,q)=q^n$ and $P_c(\GL_n(\bC),q)=q^{n^2}(q\inv)_n$.
\end{remark}

Let us now determine the Poincar\'e series of the \COHA module $\cM_\bw$.

\begin{proposition}\label{char module}
We have
$$Z(\M_{\bw},x,q)
=A_Q(x,q)\cdot S_{-2\bw}A_Q(x,q)\inv,
$$
where $S_{\bw}(x^\bd)=q^{\oh\bw\cdot\bd}x^\bd$.
\end{proposition}
\begin{proof}
It follows from \cite[Theorem 5.2]{engel_smooth} that
$$\sum_\bd (-q^\oh)^{\hi(\bd,\bd)}P_c(\Hilb_{\bd,\bw},q)x^\bd
=S_{2\bw}A_Q(x,q)\cdot A_Q(x,q)\inv.$$
As $\M_{\bw,\bd}=H^\BM(\Hilb_{\bd,\bw})[\hi(\bd,\bd)-2\bw\cdot\bd]$, we obtain
$$Z(\M_{\bw})
=\sum_\bd (-q^{\oh})^{\hi(\bd,\bd)-2\bw\cdot \bd}
P_c(\Hilb_{\bd,\bw},q)x^\bd
=A_Q(x,q)\cdot S_{-2\bw}A_Q(x,q)\inv.
$$
\end{proof}

We define \idef{(refined) DT invariants} $\Om_\bd(q)$ of the quiver $Q$ by the formula
\begin{equation}\label{DT1}
Z(\cH_Q,x,q\inv)=A_Q(x,q\inv)=\Exp\rbr{\frac{\sum_\bd
(-1)^{\hi(\bd,\bd)}\Om_\bd(q\inv) x^\bd}{1-q}}.
\end{equation}
By a theorem of Efimov \cite{MR2956038}, we have $\Om_\bd(q)\in\bN[q^{\pm\oh}]$.
We shall give a new proof of this result in Theorem~\ref{th:DT}.
The above formula for the Poincar\'e series of the \coha module can be written in the form
\begin{equation}\label{DT2}
Z(\M_{\bw},x,q\inv)=
\Exp\rbr{\sum_\bd\frac{1-q^{\bw\cdot \bd}}{1-q}(-1)^{\hi(\bd,\bd)}\Om_\bd(q\inv)x^\bd}.
\end{equation}

\section{Vertex algebras and conformal algebras}
\label{sec:VA}

Our goal in this section is to offer a detailed recollection of necessary definitions and results from the theory of vertex algebras and vertex Lie algebras, adapted to the closed symmetric monoidal category $\cC = \Vect^L$ of $L$-graded vector spaces with the braiding arising from a symmetric bilinear form, as defined in \S\ref{sec:graded}.
In particular, we define vertex Lie algebras as a particular case of the general framework of conformal algebras.
We however choose to not use the terminology
``Lie conformal algebras'' since for some readers this would hint at the presence of important conformal symmetries (such as the Virasoro algebra) included as a subalgebra, which is not the case for algebras we consider.
Our exposition merges material from many different sources, and in particular is inspired by the textbooks~\cite{MR1417941,frenkel_vertex,rosellen_vertex}.

\subsection{Graded vertex algebras}
In this section we will introduce $L$-graded vertex algebras.
We start with recalling the classical definition of a vertex superalgebra, and then explain its generalization to the $L$-graded case. 

\subsubsection{Classical vertex (super)algebras}\label{sec:vertex_alg}

Let $\Vect^{\bZ_2}$ be the category of $\bZ_2$-graded vector spaces (also called super vector spaces),
equipped with a symmetric monoidal category structure using the Koszul sign rule \eqref{koszul}.
Given a $\bZ_2$-graded vector space $V=V_0\oplus V_1$ 
and $a\in V_p$, we call $p(a)=p\in\bZ_2$ the \idef{parity} of $a$. 
The $\bZ_2$-graded vector space $\lEnd(V)=\lHom(V,V)$ is equipped with the Lie bracket
\begin{equation}\label{bracket}
[a,b]=ab-(-1)^{p(a)p(b)}ba,\qquad
a,b\in\lEnd(V).
\end{equation}
Define the space of \idef{fields}
$$\cF(V)=\sets{\sum_{n\in\bZ} a(n)z^{-n-1}\in\lEnd(V)\pser{z^{\pm1}}}
{\forall v\in V,\, a(n)v=0\text{ for }n\gg0}.
$$
This means that, for every $\ba(z)\in\cF(V)$ and $v\in V$, we have $\ba(z)v\in V\lser z$.
We equip $\cF(V)$ with the $\bZ_2$-grading, where $\ba(z)=\sum_n a(n)z^{-n-1}$ has parity $p$ if $a(n)$ has parity $p$ for all $n\in\bZ$.
We say that two fields $\ba(z),\bb(z)\in \cF(V)$ are \idef{(mutually) local} if
\begin{equation}\label{locality}
(z-w)^n[\ba(z),\bb(w)]=0,\qquad n\gg0.
\end{equation}

A \idef{vertex (super)algebra} is a triple $(V,Y,\one)$, where 
$V$ is a $\bZ_2$-graded vector space,
$Y\colon V\to \cF(V)$ is a linear map, and $\one\in V_0$ is an element called \idef{vacuum},
such that, for all $a,b\in V$,
\begin{enumerate}
\item $Y(a,z)=\sum_{n}a(n)z^{-n-1}$ has the same parity as $a$.
\item $Y(\one,z)=\id_V$, the identity operator.
\item 
$a(n)\one=0$ for $n\ge0$ and $a(-1)\one=a$.
\item $[T,Y(a,z)]=\dd_z Y(a,z)$, where $T\in\End(V)$ is defined by $T(a)=a(-2)\one$ and $\dd_z=\frac{\dd}{\dd z}$.
\item $Y(a,z)$ and $Y(b,z)$ are local.
\end{enumerate}

\begin{remark}
One can show that, for any $a\in V$,
$$Y(Ta,z)=[T,Y(a,z)]=\dd_z Y(a,z),\qquad
Y(a,z)\one=e^{zT}a.$$
Therefore
$$T(a(n)b)=(Ta)(n)b+a(n)(Tb),\qquad (Ta)(n)=-n a(n-1).$$
\end{remark}

\subsubsection{Graded vertex algebras}\label{graded VA}

In the remaining part of section of \S\ref{sec:VA}, we shall mostly work in the category $\cC^{\oh\bZ}\iso\Vect^{L\xx\oh\bZ}$ (the corresponding $\oh\bZ$-degrees will be called \idef{weights}).
It has a symmetric monoidal category structure induced from that of $\cC$ (without any additional signs coming from weights).
Consider an object $V=\bop_{n\in\oh\bZ}V_n\in \cC^{\oh\bZ}$ bounded below, meaning that $V_n=0$ for $n\ll0$. We define the space of fields
$$\cF(V)=\lEnd(V)\pser{z^{\pm1}}\in\cC^{\oh\bZ},$$
where $z$ has $L$-degree zero and weight $-1$.
Note that if $\sum_n a(n)z^{-n-1}\in\cF(V)$ has weight~$k$, then $a(n)$ has weight $k-n-1$.
This implies that, for any $b\in V$, we have 
$$a(n)b=0,\qquad n\gg0,$$
by the assumption on $V$.
Locality of fields is defined in the same way as in \eqref{locality}, using the bracket on $\lEnd(V)$.
We define an $L$-graded vertex algebra to be a triple $(V,Y,\one)$, where 
$V\in\cC^{\oh\bZ}$ is bounded below,
$Y\colon V\to \cF(V)$ is a linear map, and $\one\in V_{0}$, such that, for all $a,b\in V$,
\begin{enumerate}
\item $Y$ preserves degrees, meaning that, for $Y(a,z)=\sum_na(n)z^{-n-1}$,
the operator
$a(n)$ has the same $L$-degree as $a$ and the weight $\wt(a)-n-1$.
\item $Y(\one,z)=\id_V$.
\item 
$a(n)\one=0$ for $n\ge0$ and $a(-1)\one=a$.
\item $[T,Y(a,z)]=\dd_z Y(a,z)$, where $T\in\lEnd(V)$ is defined by $T(a)=a(-2)\one$.
\item $Y(a,z)$ and $Y(b,z)$ are local.
\end{enumerate}

Note that the map $T\colon V\to V$ has $L$-degree zero and weight $1$.
Note also that the map $Y\colon V\to \cF(V)$ is encoded by the products 
$$(n)\colon V\ts V\to V,\qquad
a\ts b\mto a(n)b,\qquad
n\in\bZ.$$
By definition, a morphism $f\colon V\to W$ between two graded vertex algebras is a degree zero linear map that preserves the vacuum and all these products.
A module over a graded vertex algebra $V$ is a graded vector space $M$ equipped with a degree zero linear map $Y_M\colon V\to \cF(M)$ that preserves products, maps $\one$ to $\id_M$ and maps $V$ to a subspace of mutually local fields. 

\begin{remark}
For a lattice $L$ equipped with a symmetric bilinear form $(\cdot,\cdot)$, it is conventional to use the braiding on $\Vect^{L}$ given by the Koszul signs $(-1)^{p(\al)p(\beta)}$, where $p(\al)=\frm(\al,\al)\pmod 2$. 
This is justified by the desire to work in the familiar setting of vertex (super)algebras.
However, this convention requires one to introduce
certain non-canonical sign twists \kac[\S5.4] (needed to ensure mutual locality of vertex operators in lattice vertex algebras), very similar to the sign twists from Remark \ref{rm:super}. 
Our convention for the braiding in $\Vect^L$ (see \S\ref{sec:graded}) allows one to avoid any non-canonical choices both in the case of \COHA{s} and in the case of vertex algebras.
\end{remark}

\subsubsection{External products}\label{sec:external}

Let $V\in\cC^{\oh\bZ}$ be bounded below, and let $\ba\in \cF(V)_\al$, $\bb\in \cF(V)_\beta$
be two fields (here $\al,\beta\in L$).
For any $n\in\bZ$, we define 
\begin{equation}
[\ba(z)\xx_n\bb(w)]
=(z-w)^n\ba(z)\bb(w) 
-(-1)^{(\al,\beta)}(-w+z)^n \bb(w)\ba(z)
\end{equation}
and consider the binary operation $\eprod n$ on $\cF(V)$ defined by
\begin{equation}\label{ext prod2}
(\ba\eprod n\bb)(w)
=\Res_z[\ba(z)\xx_n\bb(w)].
\end{equation}
We have $\ba\eprod n\bb\in \cF(V)$ by \cite[Lemma 3.1.4]{li_local}.
By Dong's lemma \cite[Proposition 3.2.7]{li_local}, if $\ba,\bb,\bc\in \cF(V)$ and $\ba,\bb$ are local with $\bc$, then $\ba\eprod n\bb$ is local with $\bc$, for all $n\in\bZ$.

The coefficients of the field $\ba\eprod m\bb$ can be written explicitly as 
\begin{equation}
(\ba\eprod m\bb)(n)
=\sum_{k\ge0}(-1)^k\binom mk 
\Big(a(m-k)b(n+k)
-(-1)^{(\al,\beta)+m}b(n+m-k)a(k)\Big).
\end{equation}
In particular, for $m\ge0$, we have
\begin{equation}\label{ext prod coeff1}
(\ba\eprod m\bb)(n)=
\sum_{k=0}^m(-1)^k\binom mk[a(m-k),b(n+k)].
\end{equation}
For $m=-1$, we obtain
$$(\ba\eprod {-1}\bb)(n)
=\sum_{k\ge0}
\Big(a(-k-1)b(n+k)
+(-1)^{(\al,\beta)}b(n-k-1)a(k)\Big).$$

\begin{remark}
The field
\begin{equation}
(\ba\eprod {-1}\bb)(z)=\ba_+(z)\bb(z)+(-1)^{(\al,\beta)}\bb(z)\ba_-(z),
\end{equation}
is the so called \emph{normally ordered product} $\nop{\ba(z)\bb(z)}$ of the fields $\ba$ and $\bb$. Here  $f_+(z)=\sum_{n\ge0}f_nz^n$ and $f_-(z)=\sum_{n<0}f_nz^n$, for $f(z)=\sum_n f_nz^n$.
\end{remark}

\begin{lemma}\label{prodcuts}
If $V$ is a vertex algebra, then $Y\colon V\to \cF(V)$ preserves products.
\end{lemma}
\begin{proof}
By the Jacobi identity \fren[8.8.30], for any $a,b\in V$, we have
\begin{equation}\label{jacobi}
[Y(a,z)\xx_n Y(b,w)]=\Res_u u^nw\inv Y(Y(a,u)b,w)
\de\rbr{\frac{z-u}{w}}
.
\end{equation}
Applying $\Res_z\de\rbr{\frac{z-u}{w}}=w$,
we obtain
$$\Res_z[Y(a,z)\xx_n Y(b,w)]
=\Res_u u^n Y(Y(a,u)b,w)=Y(a(n)b,w).$$
This implies the statement.
\end{proof}

For any $\ba\in \cF(V)$ and $n\in\bZ$, we obtain
$\ba\eprod n\in\lEnd(\cF(V))$ and we define
$$Y:\cF(V)\to\lEnd(\cF(V))\pser{u^{\pm1}},\qquad
Y(\ba,u)=\sum_{n\in\bZ}\ba\eprod nu^{-n-1}.$$
Let $\id_V\in \lEnd(V)\sbs \cF(V)$ be the identity operator
and $T=\dd_z\in\lEnd(\cF(V))$.
Then, for any $\ba\in \cF(V)$, we have \cite[Lemmas 3.1.6-3.1.7]{li_local}
\begin{enumerate}
\item $Y(\id_V,u)\ba=\ba$.
\item $Y(\ba,u)\id_V=e^{uT}\ba$.
\item $Y(T\ba,u)=[T,Y(\ba,u)]=\dd_u Y(\ba,u)$.
\end{enumerate}
The last property implies $(T\ba)\eprod n \bb=-n\ba\eprod{n-1}\bb$.
In particular, $$\ba\eprod{-2}\id_V=(T\ba)\eprod{-1}\id_V=T\ba.$$

A subspace $W\sbs \cF(V)$ is called a \idef{field algebra} if it is closed under products $\eprod n$ and contains ~$\id_V$.
It is called a \idef{local field algebra} if it consists of mutually local fields.
In this case, consider the restriction
$$Y:W\to \lEnd(W)\pser{u^{\pm1}},\qquad \ba\mto\sum_{n\in\bZ} \ba\eprod n u^{-n-1}.$$
By the locality assumption, we have
$[\ba(z)\xx_n\bb(w)]=0$, 
hence $\ba\eprod n \bb=0$, for $\ba,\bb\in W$ and $n\gg0$.
Therefore $Y(\ba,u)$ is a field and we obtain
a linear map $Y:W\to \cF(W)$.
The triple $(W,Y,\id_V)$ is a vertex algebra by \cite[Theorem 3.2.10]{li_local}.
For any subset $S\sbs \cF(V)$, consisting of mutually local (homogeneous) fields, let $\ang S$ be the smallest subspace of $\cF(V)$ containing $S\cup\set{\id_V}$ and closed under all products~$\eprod n$.
Then by Dong's lemma $\ang{S}$ is a local field algebra, and we conclude that it is a vertex algebra.

In particular, assume that $(V,Y,\one)$ is a vertex algebra.
The map $Y\colon V\to \cF(V)$ is injective as $a(-1)\one=a$.
Let $W=Y(V)\sbs \cF(V)$.
Then $W$ consists of mutually local fields by the definition of a vertex algebra.
It is closed with respect to products by Lemma \ref{prodcuts} and it contains $\id_V=Y(\one,z)$.
Therefore $W$ is a vertex algebra and $Y\colon V\to W$ is an isomorphism of vertex algebras.

\subsection{Conformal algebras}\label{sec:conformal}
In this section we will introduce the notion of conformal algebras, closely related to the notion of vertex algebras.
For more details on this subject see \eg \cite{MR1670692,MR1700512}.

\subsubsection{Definition of conformal algebras}
We define a \idef{conformal algebra} to be an object $C\in\cC^{\oh\bZ}$
equipped with a linear map $\dd\colon \fC\to\fC$ (of $L$-degree zero and weight $1$) and bilinear operations $\eprod n\colon\fC\ts\fC\to\fC$ (of $L$-degree zero and weight $-n-1$), for $n\in\bN$,
such that for all $a,b\in \fC$ we have
\begin{enumerate}
\item $a\eprod n b=0$ for $n\gg0$.
\item $\dd(a\eprod n b)=(\dd a)\eprod n b+a\eprod n (\dd b)$.
\item $(\dd a)\eprod n b=-n a\eprod {n-1} b$.
\end{enumerate}

Let us recall a fundamental class of examples of conformal algebras obtained via a procedure reminiscent of the construction of external products $\eprod n$ in \S\ref{sec:external}. Suppose that $\fL\in\cC^{\oh\bZ}$ is an algebra, by which we mean a graded vector space equipped with a bilinear operation 
 $$
[-,-]\colon\fL\ts\fL\to\fL 
 $$
(despite the suggestive notation, we do not assume this operation to satisfy any of the properties of a Lie bracket). 
We define \idef{external products} $\eprod n$
on the space $\fL\pser{z^{\pm1}}$ (where $z$ has $L$-degree zero and weight $-1$) by the formula 
(\cf \eqref{ext prod2})
\begin{equation}\label{ext prod conf}
\ba(z)\eprod n\bb(z)
=\Res_w(w-z)^n[\ba(w),\bb(z)],\qquad n\ge0,
\end{equation}
for $\ba(z)=\sum_{n\in\bZ}a(n)z^{-n-1}$ and $\bb(z)=\sum_{n\in\bZ}b(n)z^{-n-1}$ in $\fL\pser{z^{\pm1}}$.
The coefficient of $z^{-m-1}$ in this series equals (\cf \eqref{ext prod coeff1})
\begin{equation}
(\ba\eprod n\bb)(m)=
\sum_{k=0}^n(-1)^k\binom nk[a(n-k),b(m+k)].
\end{equation}
We also define $\dd=\frac{\dd}{\dd z}\colon \fL\pser{z^{\pm1}}\to\fL\pser{z^{\pm1}}$, so that the map $\dd$ and the products $\eprod{n}$ satisfy the axioms (2) and (3) above.
Note that the axiom (1) is satisfied by the fields $\ba,\bb\in\fL\pser{z^{\pm1}}$ if they are \idef{(mutually) local}, meaning that
$$(z-w)^n[\ba(z),\bb(w)]=(z-w)^n[\bb(w),\ba(z)]=0,\qquad n\gg0.$$
In this case we have \kac[\S2.2]
\begin{equation}\label{bracket from prod}
[a(m),b(n)]=\sum_{k\ge0}\binom mk (\ba\eprod k\bb)(m+n-k).
\end{equation}
We see that if $\fC\sbs \fL\pser{z^{\pm1}}$ is a subspace consisting of pairwise mutually local elements which is closed under all external products and the map $\dd$, then $\fC$ is a conformal algebra.

\subsubsection{Coefficient algebras of conformal algebras}
\label{sec:coeff algebra}
For a conformal algebra $(\fC,\dd,\eprod n)$, we shall now define the \idef{coefficient algebra} $\Coeff(\fC)$.
Note that we have 
$$(\dd a)\eprod0b=0,\qquad
a\eprod 0(\dd b)=\dd(a\eprod 0b),$$
hence $\dd C\sbs C$ is an ideal \wrt the product $\eprod0$ and we can equip $C/\dd C$ with the bracket
\begin{equation}\label{zero prod}
[a,b]=a\eprod0b.
\end{equation}
In particular, we can apply this construction
to the \idef{affinization} of $C$ \kac[\S2.7],
which is defined to be $\tl C=C[t^{\pm1}]$ equipped with the derivation $\tl\dd=\dd\ts1+1\ts\dd_t$ and the products
\begin{equation}
(af)\eprod n(bg)=\sum_{k\ge0}
a\eprod{n+k}b\cdot (\tp{\dd_t}kf)g,\qquad
a,b\in C,\, f,g\in\bQ[t^{\pm1}],
\end{equation}
where $\tp{\dd_t}k=\frac1{k!}\dd_t^k$.
Now we define the algebra of coefficients $$\fL=\Coeff(C)=\tl C/\tl\dd\tl C$$
with the bracket given by $\eprod0$.
More explicitly, let us write $C(n)=Ct^n\sbs\tl C$ and $a(n)=at^n\in C(n)$, for $a\in C$, $n\in\bZ$.
Then $\fL$ is
equal to the quotient of $\tl C=\bigoplus_{n\in \bZ}C(n)$ by the subspace generated by
\begin{equation}\label{coeff rel}
(\dd a)(n)+na(n-1),\qquad a\in \fC,\, n\in\bZ.
\end{equation}
The operation $[-,-]$ given by 
(\cf \eqref{bracket from prod})
\begin{equation}\label{bracket from prod2}
[a(m),b(n)]=
(at^m)\eprod0 (bt^n)=
\sum_{k\ge0}\binom mk (a\eprod kb)(m+n-k),
\end{equation}
descends from $\tl C$ to $\fL$.
We will usually denote the image of $a(n)\in\tl C$ in $\fL$ by $a(n)$.
The map
$$\dd:\fL\to\fL,\qquad
\dd(a(n))=(\dd a)(n)=-na(n-1)$$
is a derivation of the algebra $\fL$.
Let
$$\fL(n)=\im\Big(\fC(n)\hookrightarrow \tl C\to \fL\Big).$$
From \eqref{coeff rel} we obtain that $\fL(n-1)$ embeds into $\fL(n)$ for $n\ne 0$.
Therefore we have increasing filtrations
$$\fL(0)\sbs \fL(1)\sbs\dots\sbs \fL_+:=\sum_{n\ge0}\fL(n),\qquad
\dots\sbs \fL(-2)\sbs \fL(-1)=: \fL_-.$$
All of these subspaces are preserved by the derivation $\dd$.
It is clear from \eqref{bracket from prod2} that
$$[\fL(m),\fL(n)]\sbs\fL(m+n)$$
for $m,n\ge0$ or $m,n<0$.
In particular, $\fL_\pm$ are subalgebras of $\fL$.
Moreover, we have a direct sum decomposition \cite{MR1670692,MR1700512}
\begin{equation}\label{direct sum}
\fL=\fL_+\oplus \fL_-
\end{equation}
of the vector space $\fL$, and an isomorphism of vector spaces
$$\fC\to \fL_-=\fL(-1),\qquad
a\mto a(-1).$$
Under the isomorphism $C\iso\fL_-$, the bracket on $C$ is given by
\cite[4.14]{MR1670692}
\begin{equation}
[a,b]=\sum_{k\ge0}(-1)^k\tp\dd{k+1}(a\eprod kb).
\end{equation}
Note that the map $C\to\tl C$, $a\mto a(0)$,
is a homomorphism of conformal algebras, hence it induces a morphism of algebras
\begin{equation}
C/\dd C\to \fL(0)\sbs\fL=\tl C/\tl\dd\tl C,\qquad a+\dd C\mto a(0).
\end{equation} 
By \ro[Prop.~1.3], the map $C\to \fL(0),\,a\mto a(0),$ has the kernel $\dd C+\sum_{k\ge1}{\ker\dd^k}$, hence $C/\dd C\to\fL(0)$ is an isomorphism if $\dd$ is injective.
It is important to note that the brackets on $C$ and $C/\dd C$ are not related.

The map 
$$\vi\colon \fC\to \fL\pser{z^{\pm1}},\qquad 
a\mto \sum_{n\in\bZ}a(n)z^{-n-1}$$
is a homomorphism of conformal algebras (meaning that it preserves the products and the derivation).
Since $a\mapsto a(-1)$ is an isomorphism between $\fC$ and $\fL(-1)$, this map is injective.

The coefficient algebra $\fL=\Coeff(\fC)$ satisfies the following universal property: if $\fA$ is an algebra and $\psi\colon \fC\to \fA\pser{z^{\pm1}}$ is a morphism of conformal algebras,
then there exists a unique algebra morphism $\varrho\colon \fL\to \fA$ (given by $a(n)\mto\Res_z z^n\psi(a)$) such that the diagram 
\begin{ctikzcd}
&\fC\dlar["\phi"']\drar["\psi"]\\
\fL\pser{z^{\pm1}}\ar[rr,"\varrho"]&&\fA\pser{z^{\pm1}}
\end{ctikzcd}
is commutative. 

\subsubsection{Vertex Lie algebras and their enveloping vertex algebras}
\label{lie conformal}
Using coefficient algebras, one can define particular types of conformal algebras. 
In particular, a conformal algebra $\fC$ is called a 
\idef{\LieCon algebra} (or Lie conformal algebra)
if the algebra $\Coeff(\fC)$ is a Lie algebra.
Equivalently, this means that, for $a,b\in \fC$ and $m,n\in\bN$, we have
\begin{enumerate}
\item $\sbr{a\eprod m,b\eprod n}=\sum_{k=0}^m\binom mk(a\eprod k b)\eprod{m+n-k}$ (Jacobi identity), or equivalently\\
$(a\eprod mb)\eprod n=\sum_{k=0}^m(-1)^k\binom mk
\sbr{a\eprod{m-k},b\eprod{n+k}}$.
\item $a\eprod nb=(-1)^{1+\hi(\al,\beta)}\sum_{k\ge0}
(-1)^{n+k}\frac{\dd^k}{k!}b\eprod{n+k}a$, for $a\in\fC_\al,b\in\fC_\beta$ (anti-commuta\-tivity).
\end{enumerate}
Under these conditions, the bracket \eqref{zero prod} defined on $C/\dd C$ is a Lie bracket, see \eg \kac[Rem.~2.7a].
Let $\fC$ be a \LieCon algebra and $\fL=\Coeff(\fC)$ be its coefficient algebra.
We define the \idef{Verma module}
$$V:=U(\fL)\ts_{U(\fL_+)}\bQ\iso U(\fL_-),$$
where \bQ is equipped with the trivial $\fL_+$-module structure.
Let $\one\in V$ be the cyclic vector of this module.
One can show that, for any $a\in \fC$ and $v\in V$, we have $a(n)v=0$ for $n\gg0$.
This implies that $\vi(a)=\sum_n a(n)z^{-n-1}\in \fL\pser{z^{\pm1}}$ induces a field $\psi(a)\in\cF(V)$.
The map $\psi:\fC\to \cF(V)$ is injective as the map $\fC\iso \fL_-\to\lEnd(V)$ is injective.

The image of $\psi\colon \fC\to \cF(V)$ consists of mutually local fields and generates a vertex algebra $W\sbs\cF(V)$.
By \ro[\S2.4], the vertex algebra $W$ has a structure of an $\fL$-module and the map $W\to V$, $\ba(z)\mto a(-1)\one$, is an isomorphism of $\fL$-modules.
We shall call the vertex algebra $W\iso V$ the \idef{universal enveloping vertex algebra} of $\fC$, and denote it by $\cU(C)$.
Note that we have 
a commutative diagram of graded vector spaces
\begin{ctikzcd}
C\rar["\sim"]\dar[hook]&\fL_-\dar[hook]\\
\cU(C)\rar["\sim"]&U(\fL_-)
\end{ctikzcd}

If $V$ is a vertex algebra, then $V$ equipped with the products $(n)\colon V\ts V\to V$, $n\ge0$, and the map $\dd=T\colon V\to V$ is a \LieCon algebra.
By \cite[Theorem 5.5]{MR1670692}, the corresponding forgetful functor from the category of vertex algebras to the category of \LieCon algebras has a left adjoint functor which is precisely the universal enveloping algebra functor $C\mto\cU(C)$.

\oper{Coalg}
\subsection{Vertex bialgebras}
Let $\cC$ be a symmetric monoidal category and $(V,\De,\eps)$ be a (counital) coalgebra in $\cC$.
We define the set of \idef{group-like elements}
\begin{equation}
G(V)=\Hom_{\Coalg}(\one,V),
\end{equation}
where $\one\in\cC$ is the unit object equipped with the canonical structure of a counital coalgebra.
We say that $V$ is \idef{connected} if $G(V)$ contains exactly one element.

For example, let $L=\bZ^I$ be equipped with a symmetric bilinear form and let $\cC=\Vect^L$ be the corresponding symmetric monoidal category.
Let $V\in\cC$ be a coalgebra with degrees concentrated in $\bN^I\sbs L$ and $\dim V_0=1$.
Then $V$ is connected.

Following \cite[\S4]{li_smash}, we define a \idef{(local) vertex bialgebra} to be a vertex algebra $V$ equipped with a coalgebra structure $(V,\De,\eps)$ such that $\De:V\to V\ts V$ and $\eps:V\to\bQ$ are homomorphisms of vertex algebras.
We say that a vertex bialgebra $V$ is cocommutative if its coproduct is cocommutative.
We define the subspace of \idef{primitive elements}
\begin{equation}
P(V)=\sets{x\in V}{\De(x)=x\ts 1+1\ts x}.
\end{equation}
An important family of cocommutative connected vertex bialgebras arises from vertex Lie algebras \cite{li_smash,han_cocommutative}.

\begin{proposition}\cite[\S4]{li_smash}
\label{li-coprod}
Let $C$ be a vertex Lie algebra and $V=\cU(C)$ be its universal enveloping vertex algebra.
Then there exists vertex algebra homomorphisms 
$$\De:V\to V\ts V,\qquad\eps:V\to\bQ,$$
uniquely determined by
\begin{gather*}
\De(\one)=\one\ts\one,\qquad \eps(\one)=1,\\
\De(x)=x\ts\one+\one\ts x,\qquad \eps(x)=0\qquad\forall x\in C.
\end{gather*}
The vertex algebra $V$ equipped with $\De$ and $\eps$ is a cocommutative vertex bialgebra.
\end{proposition}

\begin{proposition}[{\cite[Prop.~3.9]{han_cocommutative}}]
\label{li-coprod2}
Let $C$ be a vertex Lie algebra, 
$\fL=\fL_-\oplus\fL_+$ be its algebra of coefficients and $V=\cU(C)=U(\fL)\ts_{U(\fL_+)}\bQ$ be its universal enveloping vertex algebra.
Then the canonical isomorphism $V\iso U(\fL_-)$ is an isomorphism of coalgebras.
Moreover, the coalgebra $V$ is connected and $P(V)=C$.
\end{proposition}

Conversely, we can associate a vertex Lie algebra with a vertex bialgebra.

\begin{proposition}[{\cite[Prop.~4.8]{han_cocommutative}}]
Let $V$ be a vertex bialgebra.
Then $P(V)$ is a vertex Lie subalgebra of $V$.
\end{proposition}

The next result is an analogue of the Milnor-Moore theorem for vertex bialgebras.

\begin{theorem}[{\cite[Theorem ~4.13]{han_cocommutative}}]
Let $V$ be a cocommutative connected vertex bialgebra.
Then there is a canonical isomorphism $\cU(P(V))\iso V$ of vertex bialgebras.
\end{theorem}

\subsection{Free vertex algebras}\label{sec:freevertex}

As was observed already in \cite[\S4]{MR843307}, the locality axiom contains a quantifier that does not allow one to define the free vertex algebra on a given set of generators, but if one restricts to some specific orders of locality, it is possible to define such a universal object; this has been done in \cite{MR1700512}. In this section we recall the definition of the locality order and explain how free vertex algebras of given non-negative locality are constructed. 

\subsubsection{Locality order}

\begin{definition}\label{def:locality order}
Let $V\in\cC^{\oh\bZ}$ be a bounded below object.
We say that two fields $\ba\in \cF(V)_\al$ and $\bb\in\cF(V)_\beta$ are (mutually)
\idef{local of order $N\in\bZ$} if
\begin{equation}\label{locality1}
(z-w)^N\ba(z)\bb(w) 
-(-1)^{\frm(\al,\beta)}(-w+z)^N \bb(w)\ba(z)
=0,
\end{equation}
where the binomial expansion convention is used.
\end{definition}

We note that for $N\ge 0$ the binomial expansion is finite and the locality condition becomes simply 
$$(z-w)^N[\ba(z),\bb(w)]=0,$$
being in agreement with the definition of locality (of unspecified order) in \S\ref{sec:vertex_alg}. 

\begin{lemma}\label{lm:loc equiv}
Let $V$ be a vertex algebra and $a,b\in V$.
Then $Y(a,z),Y(b,z)$ are local of order $N$ if and only if
$a(n)b=0$ for all $n\ge N$.
\end{lemma}
\begin{proof}
Let $\ba(z)=Y(a,z)$ and $\bb(z)=Y(b,z)$.
If $[\ba(z)\xx_N \bb(w)]=0$, then 
$[\ba(z)\xx_n \bb(w)]=0$
for all $n\ge N$.
This implies that $(\ba\eprod n\bb)(z)=0$.
As $Y\colon V\to \cF(V)$ is injective and preserves products, we conclude that $a(n)b=0$ for all $n\ge N$.
Conversely, assume that $a(n)b=0$ for all $n\ge N$.
By the Jacobi identity \eqref{jacobi}
$$[\ba(z)\xx_N \bb(w)]
=\Res_u u^Nw\inv Y(\ba(u)b,w)\de\rbr{\frac{z-u}{w}}
.$$
Under our assumption, all powers of $u$ on the right are non-negative and we conclude that 
$[\ba(z)\xx_N\bb(w)]=0$.
\end{proof}

Using the binomial expansion convention in Equation \eqref{locality1} and extracting the coefficient of $z^{N-m-1}w^{-n-1}$ , one can obtain an equivalent collection of equations 
\begin{equation}\label{locality2}
\sum_{k\ge0}(-1)^k\binom Nk a(m-k)b(n+k)-
(-1)^{\frm(\al,\beta)}\sum_{k\le N}(-1)^{k}\binom N{N-k} b(n+k)a(m-k)=0
\end{equation}
for all $m,n\in\bZ$. 
If $N\ge0$, then the locality conditions can be written in the form
\begin{equation}\label{locality3}
\sum_{k=0}^N(-1)^k\binom Nk[a(m-k),b(n+k)]=0.
\end{equation}

\subsubsection{Free vertex algebras}
Let $I$ be a set equipped with a symmetric function $\cN\colon I\xx I\to \bZ$ called the \idef{locality function} or just \idef{locality}
and a map $\deg=(\deg_L,\wt):I\to L\xx\oh\bZ$.
We consider the category $\Ver_{\cN,\deg}$ of graded vertex algebras $V\in\cC^{\oh\bZ}=\Vect^{L\xx\oh\bZ}$ equipped with a map $s\colon I\to V$ such that

\begin{enumerate}
\item 
The image of $s$ generates $V$ (as a vertex algebra).
\item
For each $i\in I$, the element $s_i$ 
has degree $\deg(i)$.
\item
For all $i,j\in I$, the fields $Y(s_i,z)$ and $Y(s_j,z)$ are local of order $\cN(i,j)$. Equivalently, $s_i(n) s_j=0$ for all $n\ge \cN(i,j)$.
\end{enumerate}

Analogously to \roi, one proves that the category $\Ver_{\cN,\deg}$ has an initial object, called the \idef{free vertex algebra \wrt locality \cN and degree $\deg$}. We shall now recall the construction of the free vertex algebra in the case of a non-negative locality function.

\subsubsection{Free vertex algebras of non-negative locality}\label{sec:free vert constr}

Suppose that the locality function $\cN$ is non-negative.  We set 
$$X=\set{i(n)\col i\in I,\, n\in\bZ}.$$
The vector space $\bQ X$ spanned by $X$ can be viewed as an object of $\cC^{\oh\bZ}$, where we set
\begin{equation}\label{degrees}
\deg_L i(n)=\deg_L(i)\in L, \qquad \wt i(n)=\wt(i)-n-1\in\oh\bZ.
\end{equation} 
In the free associative algebra $\sT(\bQ X)$, we may consider the locality relations \eqref{locality3} with $a,b\in I$ and $N=\cN(a,b)$. 
These relations are manifestly homogeneous, so they generate a homogeneous ideal $\cI\sbs \sT(\bQ X)$.
The quotient $\cB =\sT(\bQ X)/\cI$ by that ideal is an associative algebra in $\cC^{\oh\bZ}$. Let us consider the subset
$$X_+=\sets{i(n)\in X}{n\ge0},$$
and the left $\cB$-module $\P=\cB /\cB X_+$; we denote by $\one$ the cyclic element of that module.

\begin{lemma}\label{lm:restricted}
The module $\P$ is restricted, meaning that, for all $i\in I$ and $x\in \P$, we have $i(n)x=0$ for $n\gg0$.
\end{lemma}

\begin{proof}
Let $w=i_1(n_1)\dots i_p(n_p)$ be a monomial in $\mathsf{T}(\bQ X)$, and suppose that there exists $i\in I$ such that $i(n)w\notin \cB X_+$ for arbitrary large $n$. We can assume that $p$ is minimal possible.
Then 
$$i_1(k)i_2(n_2)\dots i_p(n_p)\in \cB X_+$$
for $k\gg0$.
Therefore we can assume that for any $k>n_1$, we have
$$i(n)i_1(k)i_2(n_2)\dots i_p(n_p)\in \cB X_+$$
for $n\gg0$. Let $N=\cN(i,i_1)$. Then we have a locality relation 
\begin{equation*}
\sum_{k=0}^N(-1)^k\binom Nk[i(n-k),i_1(n_1+k)]=0
\end{equation*}
in the algebra $\cB$. This relation implies that $i(n)w$ is contained in the left ideal generated by elements
$$i(n')i_1(k)i_2(n_2)\dots i_p(n_p),\qquad 
n'\ge n-N,\,k>n_1,$$
$$i(n')i_2(n_2)\dots i_p(n_p),\qquad n'\ge n-N.$$
By our assumption, these elements are contained in $\cB X_+$ for $n\gg0$.
\end{proof}

This lemma implies that for every $i\in I$, we have a field $\bi(z)=\sum_{n\in\bZ}i(n)z^{-n-1}\in \cF(\P)$. These fields are mutually local by construction.
Let us define a derivation $T$ of $\mathsf{T}(\bQ X)$ by the formula $T i(n)=-ni(n-1)$. 

\begin{lemma}\label{lm:der T}
The derivation $T$ preserves the ideal $\cI$. In particular, $T$ induces endomorphisms of $\cB$ and $\P$.
\end{lemma}

\begin{proof}
Let us consider $\bi(z)=\sum_{n\in\bZ}i(n)z^{-n-1}$ as an element of $\mathsf{T}(\bQ X)\pser{z^{\pm1}}$. Using this notation, we see that the relations that generate $\cI$ are coefficients of $(z-w)^N[\bi(z),\bj(w)]$ with $N=\cN(i,j)$. 
Since $T\bi(z)=\dd_z\bi(z)$, we have
$$\dd_z\Big((z-w)^N[\bi(z),\bj(w)]\Big)
=N(z-w)^{N-1}[\bi(z),\bj(w)]+(z-w)^N[T\bi(z),\bj(w)]$$
and
$$\dd_w\Big((z-w)^N[\bi(z),\bj(w)]\Big)
=-N(z-w)^{N-1}[\bi(z),\bj(w)]+(z-w)^N[\bi(z),T\bj(w)].$$
Adding these equations, we obtain
$$T(z-w)^N[\bi(z),\bj(w)]=
\dd_z\Big((z-w)^N[\bi(z),\bj(w)]\Big)+\dd_w\Big((z-w)^N[\bi(z),\bj(w)]\Big),
$$ which means that locality relations are preserved by $T$. 
\end{proof}

Applying the existence theorem \kac[Theorem 4.5] to the fields $\bi(z)\in \cF(\P)$ and the endomorphism $T\in\End(\P)$, we obtain a vertex algebra structure on $\P$ with 
$$Y(i_1(n_1)\dots i_p(n_p)\one,z)
=\bi_1(z)\eprod{n_1}(\bi_2(z)\eprod{n_2}(\dots (\bi_p(z)\eprod{n_p}\id_{\P})\dots)).$$
The vertex algebra $\P$ is the free vertex algebra \wrt locality \cN and degree function $\deg$.
We will sometimes denote the element $i(-1)\one\in\P$ by $i$.
\medskip

\subsubsection{Free vertex algebras as universal enveloping algebras}
Next we will show that the free vertex algebra $\P$ of non-negative locality
can be identified with the universal enveloping algebra of a \LieCon algebra.
Let $\fL$ be the Lie algebra with generators $i(n)\in X$ subject to the locality relations \eqref{locality3}
with $a,b\in I$ and $N=\cN(a,b)$,
and let $\fL_+\sbs\fL$ be the subalgebra generated by $X_+\sbs X$.
Then we have
$$\cB\iso U(\fL),\qquad \P=\cB/\cB X_+\iso U(\fL)\ts_{U(\fL_+)}\bQ,$$
For every $i\in I$, we consider $\bi(z)=\sum_n i(n)z^{-n-1}\in \fL\pser{z^{\pm1}}$.
These elements are mutually local by definition, therefore by Dong's Lemma \cite[Prop.~ 3.2.7]{li_local}
they generate, under non-negative external products and the derivation $\dd_z$, a conformal algebra 
$$\fC\sbs \fL\pser{z^{\pm1}}$$
which is a \LieCon algebra as conditions from
\S\ref{lie conformal} are automatically satisfied.

\begin{proposition}\label{free universal enveloping}
The free vertex algebra $\P$ of non-negative locality is isomorphic to the universal enveloping vertex algebra of the \LieCon algebra $C$.
\end{proposition}
\begin{proof}
First, let us show that the algebra morphism 
$$\varrho:\Coeff(\fC)\to \fL,\qquad
\ba(z)(n)\mto\Res_z z^n\ba(z),$$
is an isomorphism (\cf \cite[Prop.~3.1]{MR1700512}).
(In fact, this \LieCon algebra is the free \LieCon algebra \wrt locality $\cN$ and degree function $\deg$.)
To construct the inverse map we need to show that elements $\bi(z)(n)$ satisfy the locality relations.
This means that for $N=\cN(i,j)$, we have
$$\sum_{k=0}^N(-1)^k\binom Nk
[\bi(z)(m-k),\bj(z)(n+k)]=0$$
in $\Coeff(C)$.
Using the formula for the bracket in $\Coeff(C)$, this equation can be written in the form
$$\sum_{k=0}^N(-1)^k\binom Nk
\sum_{l\ge0}\binom{m-k}{l}(\bi\eprod l\bj)(m+n-l)=0.$$
Note that $\bi\eprod l\bj=0$ for $l\ge N$ by the locality relations.
We claim that, on the other hand, for any 
$0\le l<N$, we have
$$\sum_{k=0}^N(-1)^k\binom Nk
\binom{m-k}{l}=0.
$$
This claim would imply the required statement.
We have
\begin{multline*}
\sum_{k,l\ge0}\binom Nk
\binom{m-k}{l}x^ky^l
=\sum_{k\ge0}\binom Nk x^k(1+y)^{m-k}
=(1+y)^m\rbr{1+\frac x{1+y}}^N\\
=(1+x+y)^N(1+y)^{m-N}.
\end{multline*}
In particular, for $x=-1$, we obtain
$$\sum_{k,l\ge0}(-1)^k\binom Nk
\binom{m-k}{l}y^l
=y^N(1+y)^{m-N}
$$
which has only powers of $y$ that are $\ge N$.
This proves the claim.

The derivation $\dd$ on $\Coeff(C)$ corresponds to the derivation $T$ on $\fL\sbs\cB$ from Lemma~\ref{lm:der T}.
Using the isomorphism $\varrho$, we obtain the direct sum decomposition 
$$\fL=\fL_+\oplus\fL_-,$$
where both subalgebras are preserved by $T$.
By \cite[\S3.2]{MR1700512}, the algebra $\fL_+\sbs\fL$ is
exactly the algebra generated by $X_+=\sets{i(n)\in X}{n\ge0}$
which we introduced earlier.

This implies that $\P=\cB/\cB X_+$ is isomorphic to the Verma module $U(\fL)\ts_{U(\fL_+)}\bQ$ as a $\cB$-module, and therefore the free vertex algebra $\P$ is isomorphic 
to the universal enveloping vertex algebra $\cU(\fC)$.
\end{proof}

\subsection{Lattice vertex algebras}
\label{lattice VA}
In this section, we recall an important class of vertex algebras called lattice vertex algebras. They play a prominent role in our work, for they can be used to construct convenient realisations of certain free vertex algebras. 

\subsubsection{Fock spaces}
We consider the $\bQ$-vector space $\fh=L\ts_\bZ\bQ$ associated to $L$, and define the Heisenberg Lie algebra $\hat\fh$ as a one-dimensional extension
$$0\to\bQ K\to\hat\fh\to\fh\ts\bQ[t^{\pm1}]\to0$$
with central $K$ and
$$[h_m,h'_n]=m\de_{m,-n}\cdot \frm(h,h')K.$$
Here and below, we define $h_n:=h\ts t^n$, for $h\in\fh$ and $n\in\bZ$. 
If we set weight degrees
$$\wt(h_n)=-n,\qquad\wt(K)=0,$$
then $\hat\fh$ becomes a Lie algebra in $\Vect^{\oh\bZ}$.

Let $\sH$ be the associative Heisenberg algebra which is the quotient of the universal enveloping algebra of $\hat\fh$ modulo the ideal generated by $K-1$.
Consider the commutative subalgebra $$\sH_+=U(\fh\ts\bQ[t])\iso S(\fh\ts\bQ[t])\sbs\sH.$$
For any $\la\in L\dual=\Hom(L,\bZ)$, let $\bQ_\la$ be the $1$-dimensional module over $\sH_+$ where, for each $h\in\fh$, $h_0$ acts by multiplication
with $\la(h)$ and $h_n$ with $n>0$ acts trivially. We define the \idef{Fock space} $V_\la$ to be the $\sH$-module
$$V_\la=\sH\ts_{\sH_+}\bQ_\la .$$
As a vector space, it is isomorphic to $S(\fh\ts t\inv\bC[t\inv])$. 
We denote the cyclic vector of $V_\la$ by $\ket\la$.
In particular, for any $\al\in L$, we may consider the linear function $\frm(\al,\cdot)\in L\dual$
and by abuse of notation, we shall denote the corresponding 
Fock space by $V_\al$ and its generator by $\ket\al$.

For any $\al\in L$, we have a field $\al(z)=\sum_{n\in\bZ}\al_nz^{-n-1}\in \cF(V_0)$.
One knows that $V_0$ has a vertex algebra structure with the vacuum $\ket 0$
and with 
$$Y(\al^{(1)}_{n_1}\dots \al^{(p)}_{n_p}\ket 0,z)
=\al^{(1)}(z)\eprod{n_1}(\al^{(2)}(z)\eprod{n_2}(\dots \al^{(p)}(z)\eprod{n_p}\id_{V_0})).$$
In particular, $Y(\al_{-1}\ket0,z)=\al(z)$.
Similarly, using $\id_{V_\la}$ in the above formula, we can equip $V_\la$ with a module structure over the vertex algebra $V_0$.

\subsubsection{Vertex operators and lattice vertex algebras}\label{lattice voa}

Let us consider the \idef{total Fock space}
$$V_L=\bop_{\al\in L}V_\al .$$
This space is automatically $L$-graded. We equip it with a weight grading for which $\wt(\ket \beta)=\oh\frm(\beta,\beta)$,
so that
\begin{equation}
\wt(\al^{(1)}_{n_1}\dots \al^{(p)}_{n_p}\ket \beta)
=\oh\frm(\beta,\beta)-\sum_{k=1}^p n_k\in\oh\bZ.
\end{equation}
Thus, $V_L$ may be considered as an object in $\cC^{\oh\bZ}$.

Let $S_\al\colon V_L\to V_L$ be the unique linear map
satisfying
\begin{enumerate}
\item $S_\al \ket\beta=\ket{\al+\beta}$.
\item $[h_n,S_\al]=0$ for $n\ne0$ and $h\in\fh$.
\end{enumerate}
Let
$z^{\frm(\al,-)}\colon V_L\to V_L\pser{z^{\pm1}}$
be the map that acts on $V_\beta$ by multiplication with $z^{\frm(\al,\beta)}$.
Note that
$$z^{\frm(\al,-)} S_\beta=z^{\frm(\al,\beta)}S_\beta z^{\frm(\al,-)}.$$

For any $\al\in L$, we let 
$$
\Ga^-_\al(z)=
\exp\rbr{-\sum_{n<0}\frac{\al_n}nz^{-n}}
,\qquad
\Ga^+_\al(z)=\exp\rbr{-\sum_{n>0}\frac{\al_n}nz^{-n}}
.$$
Using these series, we now define the \idef{vertex operator} (a field on $V_L$)
\begin{equation}
\Ga_\al(z)
=\sum_n \Ga_\al(n)z^{-n-1}
=S_\al \Ga^-_\al(z)\Ga^+_\al(z)z^{\frm(\al,-)}.
\end{equation}

Let us check that the operators $\Ga_\al(z)$ are pairwise local.

\begin{lemma}\label{VO locality}
For any $\al,\beta\in L$, we have
$$(z-w)^{-\frm(\al,\beta)}\Ga_\al(z)\Ga_\beta(w)
=(w-z)^{-\frm(\al,\beta)}\Ga_\beta(w)\Ga_\al(z).$$
Equivalently, $\Ga_\al(z)$, $\Ga_\beta(z)$ are local of order  $-\frm(\al,\beta)$.
\end{lemma}
\begin{proof}
First of all, we have
$$\Ga^+_\al(z)\Ga^-_\beta(w)
=\rbr{1-\frac wz}^{\frm(\al,\beta)}
\Ga^-_\beta(w)\Ga^+_\al(z)$$
which follows from the fact that 
$$\exp\rbr{
\sbr{\sum_{n>0}\frac{\al_n}nz^{-n},
\sum_{n<0}\frac{\beta_n}nw^{-n}}
}
=\exp\rbr{(\al,\beta)\sum_{n>0}\frac{(w/z)^n}{-n}}
=\rbr{1-\frac wz}^{\frm(\al,\beta)}.
$$
Now we obtain
\begin{multline}\label{Gamma-product}
\Ga_\al(z)\Ga_\beta(w)
=S_\al \Ga^-_\al(z)\Ga^+_\al(z)z^{\frm(\al,-)}\cdot
S_\beta \Ga^-_\beta(w)\Ga^+_\beta(w)w^{\frm(\beta,-)}\\
=z^{\frm(\al,\beta)}S_{\al+\beta}
\Ga^-_\al(z)\Ga^+_\al(z)\Ga^-_\beta(w)\Ga^+_\beta(w)
z^{\frm(\al,-)} w^{\frm(\beta,-)}
\\
=(z-w)^{\frm(\al,\beta)}S_{\al+\beta}
\Ga^-_\al(z)\Ga^-_\beta(w)\Ga^+_\al(z)\Ga^+_\beta(w)
z^{(\al,-)} w^{(\beta,-)}
\end{multline}
and similarly for $\Ga_\beta(w)\Ga_\al(z)$.
This implies the first claim of the lemma. 
For the second claim, it is enough to note that $(w-z)^{-\frm(\al,\beta)}=(-1)^{\frm(\al,\beta)}(-w+z)^{-\frm(\al,\beta)}$ and recall Definition~\ref{def:locality order}. 
\end{proof}

Note that the weight of $z$ is $-1$ and the weight of $\al_n$ is $-n$. Therefore $\al_n z^{-n}$ has weight zero and the same is true for $\Ga_\al^{\pm}(z)$.
The element $S_\al z^{\frm(\al,-)}\ket\beta=z^{\frm(\al,\beta)}\ket{\al+\beta}$ has weight $\oh\frm(\al,\al)+\oh\frm(\beta,\beta)$.
We conclude that the operator $\Ga_\al(z)$ has weight $\oh\frm(\al,\al)$ which is also the weight of $\ket\al$.
The total Fock space $V_L$ has a structure of an $L$-graded vertex algebra, called the \idef{lattice vertex algebra}, that extends the vertex algebra structure on $V_0$ and satisfies
$$Y(\ket\al,z)=\Ga_\al(z).$$

\subsubsection{Principal free vertex algebras}
\label{sec:princ}
Let $(e_i)_{i\in I}$ be a basis of the lattice $L$.
Recall that, according to Lemma \ref{VO locality}, the fields $\Gamma_i(z):=\Gamma_{e_i}(z)$ and $\Gamma_j(z):=\Gamma_{e_j}(z)$ on $V_L$ are local of order $-\frm(e_i,e_j)$.
We define the locality function
$$\cN:I\xx I\to\bZ,\qquad \cN(i,j)=-(e_i,e_j)$$
and the degree map
$$\deg:I\to L\xx\oh\bZ,\qquad \deg(i)=\rbr{e_i,\oh(e_i,e_i)}.$$
Let $\P$ be the free vertex algebra \wrt locality \cN and degree $\deg$,
which we shall also call the
\emph{principal free vertex algebra}.
By the universality property of the free vertex algebra, there is a unique morphism of graded vertex algebras
$$\P\to V_L,\qquad i\mto \ket{e_i}.$$
By \cite[Theorem~2]{MR1935501}, this morphism is injective.
Therefore the free vertex algebra $\P$ can be identified with the vertex subalgebra of $V_L$ generated by elements $\ket{e_i}$, $i\in I$.
This subalgebra was studied in \cite{MR2967107} under the name \idef{principal subalgebra}.

\section{CoHA and the free vertex algebra}\label{coha and free}

From now onwards, we restrict ourselves to the situation considered in \S\ref{sec coha}, so that
$Q$ is a symmetric quiver with the set of vertices $I$, and $L=\bZ^I$ is a lattice equipped with the Euler form~$\hi$ of the quiver $Q$.
Then $I$ parametrizes the standard basis of $L$, and we can consider the corresponding principal free vertex algebra $\P$ from \S\ref{sec:princ},
which we shall usually denote by~$\P_Q$.
The locality function of that free vertex algebra is given by $$\cN(i,j)=-\hi(e_i,e_j),$$
hence we see that $\cN(i,j)\ge0$ for $i\ne j$ and $\cN(i,i)\ge-1$. 
This property together with Lemma \ref{locality -1 and 0} implies that the construction of \S\ref{sec:free vert constr}
for free vertex algebras of non-negative locality applies.

In this section, we shall unravel a deep relationship between the principal free vertex algebra $\P_Q$ and the \COHA $\cH_Q$. In \S\ref{graded dual}, we shall see that there is an isomorphism of vector spaces between the graded dual $\P_Q\dual$ and $\cH_Q$. In \S\ref{coalg vertex}, we shall show that this isomorphism allows one to recover the algebra structure of $\cH_Q$. Finally, in \S\ref{efimov thm}, we use our results to obtain a new proof of Efimov's positivity theorem for Donaldson--Thomas invariants.

\subsection{Three incarnations}
\label{sec:three}
In this section, we shall discuss three different incarnations of the same vertex algebra $\P_Q$: 
as a principal free algebra,
as a free algebra of non-negative locality, 
and as the universal enveloping algebra of a \LieCon algebra.
From \eqref{locality3}, we know that for $\cN(i,j)\ge0$, the locality relation between $Y(i,z)$ and $Y(j,z)$ is a collection of ``honest'' Lie algebra relations (finite combinations of Lie brackets).
In our situation,
the only possible negative value of the locality function is $\cN(i,j)=-1$, in which case $i=j$ and $\hi(e_i,e_i)=1$. 

\begin{lemma}\label{locality -1 and 0}
Let $V$ be a graded vertex algebra, and suppose that $\ba\in \cF(V)_\al$ is a field with $\hi(\al,\al)\equiv 1\pmod{2}$. Then $\ba$ is local with itself of order $-1$ if and only if it is local with itself of order $0$. 
\end{lemma}

\begin{proof}
We already know that locality of order $-1$ implies locality of order $0$, so we only need to prove the reverse implication. The locality of order  $-1$ is equivalent to the relations 
\begin{equation}\label{loc-1}
\sum_{k\ge0}a(m-k)a(n+k)-
\sum_{k<0}a(n+k)a(m-k)=0
\end{equation}
for all $m,n\in\bZ$.
The locality of order  $0$ is equivalent to  
\begin{equation}\label{loc0-2}
[a(m),a(n)]=a(m)a(n)+a(n)a(m)=0,\qquad m,n\in\bZ.
\end{equation}
It remains to note that, assuming that $m\ge n$, Condition \eqref{loc-1} can be rewritten as 
$$
\sum_{k=0}^{m-n}a(m-k)a(n+k)=0,
$$
which vanishes whenever \eqref{loc0-2} holds. 
\end{proof}

It follows from Lemma \ref{locality -1 and 0} that 
if $\cN(i,i)=-1$, then the field $Y(i,z)$ is local of order $0$ with itself, so the principal free vertex algebra $\P_Q$ coincides with the free vertex algebra for the non-negative locality function $\cN_+:=\max(\cN,0)$.
In particular, even though some values of the locality function for $\P_Q$ can be negative, Proposition \ref{free universal enveloping} implies that $\P_Q$ is isomorphic to the universal enveloping vertex algebra $\cU(\fC)$ of a \LieCon algebra $\fC$.

Let us recall some notation for future reference.
We consider the sets
$$X=\sets{i(n)}{i\in I,\,n\in\bZ},\qquad
X_+=\sets{i(n)}{i\in I,\,n\ge0},$$
and define $\fL$ to be the Lie algebra generated by $X$ subject to locality relations \eqref{locality3} 
with $a,b\in I$ and $N=\cN_+(a,b)$,
and $\fL_+\sbs\fL$ to be the subalgebra generated by $X_+\sbs X$.

For every $i\in I$, we consider $\bi(z)=\sum_n i(n)z^{-n-1}\in \fL\pser{z^{\pm1}}$ and we define
$\fC\sbs \fL\pser{z^{\pm1}}$ to be the \LieCon algebra generated by these series.
We know from Proposition \ref{free universal enveloping} that $\fL\iso\Coeff(C)$ and that we have a direct sum decomposition
$$\fL=\fL_+\oplus\fL_-,$$
where $\fL_{\pm}\sbs\fL$ are Lie subalgebras closed under the derivation 
\begin{equation}
\dd:\fL\to\fL,\qquad
\dd(i(n))=-n i(n-1).
\end{equation}
We have an isomorphism of graded vector spaces (compatible with derivations)
\begin{equation}
C\xto{\ \sim\ }\Coeff(C)_-\xto{\ \sim\ }\fL_-,\qquad
\ba(z)\mto \ba(z)(-1)\mto\Res_z z\inv\ba(z).
\end{equation}
Moreover the principal free vertex algebra $\P_Q$ is isomorphic to the universal enveloping vertex algebra
of the \LieCon algebra $C$
\begin{equation}
\P_Q\iso \cU(C)=U(\fL)\ts_{U(\fL_+)}\bQ.
\end{equation}
We also have an isomorphism of graded vector spaces 
\begin{equation}\label{eq:P-L}
\P_Q\iso U(\fL)\ts_{U(\fL_+)}\bQ\iso U(\fL_-).
\end{equation}

\subsection{The graded dual of the principal free vertex algebra}\label{graded dual}
Our first goal is to exhibit a relationship between the underlying objects of $\cP_Q$ and $\cH_Q$. 
Recall from \eqref{degrees} that $\P_Q$ is equipped with the $L$-degree as well as the weight 
$$\wt(i_1(n_1)\dots i_p(n)\one)
=\sum_k \rbr{\oh \hi(e_{i_k},e_{i_k})-n_k-1}\in\oh\bZ.$$
As we are aiming to relate vertex algebras to cohomological Hall algebras, we shall implement a relationship between half-integer weights and integer cohomological degrees, defining, for $u\in\P_Q$, its (cohomological) degree to be
$$\deg(u)=-2\wt(u)\in\bZ.$$
If we superpose the $L$-degree and the degree we just defined, $\P_Q$ becomes an object of $\Vect^{L\xx \bZ}$ with finite-dimensional components $\P_{Q,\bd}^n$, for $\bd\in \bN^I$ and $n\in\bZ$; this immediately follows from the fact that the lattice realisation of free vertex algebras allows us to view $\P_Q$ as a subspace of the total Fock space $V_L$. We define the graded dual space of $\cP_Q$ by the formula 
$$\P_Q\dual
=\bop_{\bd\in\bN^{I}}\P_{Q,\bd}\dual,\qquad
\P_{Q,\bd}\dual=\lHom(\P_{Q,\bd},\bQ)
=\bop_{n\in\bZ}(\P_{Q,\bd}^{-n})\dual.
$$

Let us describe this graded dual directly. For a dimension vector $\bd\in\bN^I$, let $p=\n\bd$ and $\lb j=(j_1,\dots,j_p)$ be any sequence of vertices with $e_{j_1}+\cdots+e_{j_p}=\bd$.
We shall associate to such sequence $\lb j$ and $\ksi\in\cP_{0,\bd}\dual$ an element 
$F_{\ksi,\lb j}(z_1,\ldots,z_p) \in\bQ\lser{z_1}\dots\lser{z_p}$ given by the formula
 $$ 
F_{\ksi,\lb j}(z_1,\ldots,z_p)=\prod_{k<l}(z_k-z_l)^{\cN(j_k,j_l)}\cdot
\ang{\xi, Y(j_1,z_1)Y(j_2,z_2)\cdots Y(j_p,z_p)\one}.
 $$

\begin{lemma}\label{lm:sympoly}
For any $\ksi\in\cP_{0,\bd}\dual$, the Laurent series $F_{\ksi,\lb j}(z_1,\ldots,z_p)$ is completely symmetric under the action of $\Si_p$ permuting simultaneously the vertices $j_p$ and the variables $z_p$. Moreover, this Laurent series is a polynomial in $z_1,\ldots, z_p$. 
\end{lemma}
\begin{proof}
By the locality properties of the free vertex algebra $\P_Q$, we have
$$(z-w)^{\cN(i,j)}Y(i,z)Y(j,w)
=(w-z)^{\cN(i,j)}Y(j,w)Y(i,z),\qquad i,j\in I.$$
This shows that $F_{\ksi,\lb j}(z_1,\ldots,z_p)$ is fully symmetric under the action of $\Si_p$. 

To prove that this expression is a polynomial, we shall argue as follows. We note that since the vector $\one$ of $\P_Q$ is annihilated by each generator $i(n)$ with $n\ge 0$ (appearing in $Y(i,z)$ as the coefficient of $z^{-n-1}$), the expression $F_{\ksi,\lb j}(z_1,\ldots,z_p)$ does not contain any negative power of the last variable $z_p$. 
Multiplying by $\prod_{k<l}(z_k-z_l)^{\cN(j_k,j_l)}$ has the effect of making the result symmetric, so it would be almost completing the proof, except that we may have vertices with $\cN(i,i)=-1$, which would create negative powers from the geometric series expansion. However, this would never create negative powers of the last variable corresponding to such vertex, and since the expression is fully symmetric, there are no negative powers at all. Finally, since we are dealing with the graded dual vector space, the linear functional $\ksi$ is supported at finitely many degrees, and so our Laurent series is in fact a polynomial.  
\end{proof}

We shall now state and prove the main result of this section. For a dimension vector $\bd\in\bN^I$, we now make a choice of a concrete sequence of vertices $\lb i=(i_1,\dots,i_p)$ with $e_{i_1}+\cdots+e_{i_p}=\bd$. For that, we choose an order of $I$, use this order to identify $I$ with $\set{1,\dots,r}$, where $r=|I|$, and put
$$i_{\bd_1+\dots+\bd_{j-1}+k}=j,\qquad j\in I,\,1\le k\le \bd_j.$$
We shall use formal variables $z_{\bd_1+\dots+\bd_{j-1}+k}=x_{j,k}$ associated to this sequence of vertices.
Recall from \S\ref{sec:sguffle} that
$$\La_\bd
=\bts_{j\in I}\La_{\bd_j}
=\bQ[x_{j,k}\col 
j\in I,\, 1\le k\le \bd_j]^{\Si_\bd}$$
is a graded algebra with $\deg(x_{j,k})=2$ (corresponding to weight $-1$, \cf \S\ref{graded VA}).

\begin{proposition}\label{prop:sympoly}
We have an isomorphism of graded vector spaces
\begin{equation}
F\colon \cP_Q\dual\to \cH_Q.
\end{equation}
Specifically, for any dimension vector $\bd$, the isomorphism
$\P_{Q,\bd}\dual\to  \cH_{Q,\bd}=\La_\bd[-\hi(\bd,\bd)]$
is given by the formula
$$
\ksi\mapsto F_\ksi=\prod_{k<l}(z_k-z_l)^{\cN(i_k,i_l)}\cdot\left.
\ang{\xi, Y(i_1,z_1)Y(i_2,z_2)\cdots Y(i_p,z_p)\one}\right|_{z_{\bd_1+\dots+\bd_{j-1}+k}=x_{j,k}}.$$
\end{proposition}

\begin{proof}
We begin with noticing that, according to Lemma \ref{lm:sympoly}, the polynomial $F_\ksi$ in variables $x_{j,k}$ is symmetric with respect to the action of the group
$\Si_\bd=\prod_{i\in I}\Si_{\bd_i}$ (the action of a permutation on the vertices $i_p$ and the variables $z_p$ is the same as the action of a permutation on the variables $z_p$ if we consider a permutation which only permutes equal vertices). Thus, $F$ is a well-defined map into $\La_\bd[-\hi(\bd,\bd)]$.

Let us first check that $F$ is a map of degree zero. Recall that $Y(i,z)$ has weight $\oh\hi(e_i,e_i)$, so by our convention $\deg(u)=-2\wt(u)$, the degree of $Y(i,z)$ is equal to $-\hi(e_i,e_i)$. Therefore $F_\ksi$ has degree
$$\deg(\ksi)+2\sum_{k<l}\cN(i_k,i_l)-\sum_k\hi(e_{i_k},e_{i_k})
=\deg(\ksi)-\sum_{k,l}\hi(e_{i_k},e_{i_l})
=\deg(\ksi)-\hi(\bd,\bd).
$$
This constant shift of degrees by $-\hi(\bd,\bd)$ means exactly that the map $$F\colon \cP_{Q,\bd}\dual\to\La_\bd[-\hi(\bd,\bd)]$$ has degree zero.

Let us show that $F$ is injective. Suppose that for some $\ksi\in\cP_{Q,\bd}\dual$ we have $F_\ksi=0$. According to Lemma \ref{lm:sympoly}, this implies that    
$$\prod_{k<l}(z_k-z_l)^{\cN(j_k,j_l)}\cdot
\ang{\xi, Y(j_1,z_1)Y(j_2,z_2)\cdots Y(j_p,z_p)\one}=0$$
for any sequence of vertices with $e_{j_1}+\cdots+e_{j_p}=\bd$. 
Since $F_\ksi$ is a polynomial, its product with $\prod_{k<l}(z_k-z_l)^{-\cN(j_k,j_l)}$ (expanded according to the binomial expansion convention) is well defined, so we may conclude that 
$$\ang{\xi, Y(j_1,z_1)Y(j_2,z_2)\cdots Y(j_p,z_p)\one}=0.$$
Extracting coefficients of individual monomials $z_1^{m_1}\cdots z_p^{m_p}$, we see that $\xi$ vanishes on all vectors $j_1(n_1)j_2(n_2)\ldots j_p(n_p)\one\in \P_{0,\bd}$, so $\xi=0$. 

Let us finally show that $F$ is surjective. For that, we shall use the lattice vertex operator realisation, replacing $Y(i_k,z_k)$ by vertex operators $\Ga_{{i_k}}(z_k)=\Ga_{e_{i_k}}(z_k)$ and $\one$ by $\ket0$. Using \eqref{Gamma-product}, we can express $F_\ksi$ explicitly as
$$\bra\xi \Ga^-_{{i_1}}(z_1)\dots\Ga^-_{{i_p}}(z_p)\ket\bd|_{z_{\bd_1+\dots+\bd_{j-1}+k}=x_{j,k}} .$$
We note that the expression 
$$\Ga^-_{i_1}(z_1)\dots\Ga^-_{i_p}(z_p)\ket\bd|_{z_{\bd_1+\dots+\bd_{j-1}+k}=x_{j,k}}$$
can be written as 
 $$
\prod_{j\in I}\exp\left(\sum_{n>0}e_{j,-n}\frac{\sum_{k=1}^{\bd_j}x_{j,k}^{n}}{n}\right)\ket\bd,$$
which, after expanding the exponential, is the formal infinite sum of all basis vectors of the Fock space $V_\bd$ with coefficients being scalar multiples of the corresponding monomials in the power sum $\Si_\bd$-symmetric functions \cite{Macdonald}. Thus, we may obtain all such monomials when take linear functions $\xi$ obtained by restriction to $\P_Q$ of the basis of $V_\bd^\vee$ dual to the standard basis of $V_\bd$.
Since the power sum symmetric functions generate $\La_\bd$, the surjectivity claim follows. 
\end{proof}

Using our result, we can immediately recover the formula for the Poincar\'e series $Z(\cP_Q,x,q)$ proved in \mil[Theorem 5.3]
using combinatorial bases for free vertex algebras. 

\begin{corollary}\label{cor:char free vertex}
We have
$$Z(\cP_Q,x,q)
=\sum_{\bd\in\bN^I}
\frac{(-q^{\oh})^{\hi(\bd,\bd)}}{(q)_\bd}x^\bd.
$$
\end{corollary}
\begin{proof}
By Proposition \ref{char coha}, we have
$$Z(\cH_Q,x,q)
=\sum_{\bd\in\bN^I}
\frac{(-q^\oh)^{-\hi(\bd,\bd)}}{(q\inv)_\bd}x^\bd.$$
On the other hand by the established isomorphism we have
$Z(\cP_Q,x,q)=Z(\cP_Q\dual,x,q\inv)=Z(\cH_Q,x,q\inv)$.
\end{proof}

\subsection{The coalgebra structure of the principal free vertex algebra}\label{coalg vertex}

Our next goal is to unravel a canonical coproduct on $\P_Q$
(\cf Proposition~\ref{li-coprod}) and to relate it to the product on $\cH_Q$. 
The universal enveloping algebra $U(\fL)$ is a bialgebra,
with the canonical coproduct $\Delta$.
Because of this coproduct, the tensor product
$\P_Q\otimes\P_Q$ acquires a $U(\fL)$-module structure; note that the action of $U(\fL)$ has obvious signs arising from the braiding.  

\begin{proposition}\label{prop:coalg}
There exists a unique morphism of $U(\fL)$-modules $\delta\colon\P_Q\to \P_Q\otimes\P_Q$ for which $\delta(\one)=\one\otimes\one$.
This morphism makes $\P_Q$ a cocommutative coassociative coalgebra.
\end{proposition} 

\begin{proof}
Recall that $\P_Q$ may also be described as $U(\fL)/U(\fL)X_+$, where 
$$X_+=\sets{i(n)\in X}{n\ge0}.$$
The $U(\fL)$-module morphism condition implies that 
 \[
\delta(i(n)w)=(i(n)\otimes\Id+\Id\otimes i(n))\delta(w)
 \]
for all $w\in\P_Q$, and so it is clear that if the morphism $\delta$ exists, then it is unique, and is induced by $\Delta$. Thus, all we have to show is that $\Delta$ descends to the quotient by the left ideal generated by $X_+$, which follows from the property 
 \[
\Delta(X_+)\subset X_+\otimes \bQ + \bQ\otimes X_+.
 \]  
\end{proof}

Recall that we have an isomorphism of graded vector spaces \eqref{eq:P-L}
$$\P_Q\iso U(\fL)\ts_{U(\fL_+)}\bQ\iso U(\fL_-).$$
The above coproduct on $\P_Q$ can be identified with the canonical coproduct on $U(\fL_-)$ (\cf~ Proposition~\ref{li-coprod2}).
We are now ready to state and prove the main result of this section. 

\begin{theorem}\label{th:VOA-CoHA}
The isomorphism $F\colon \P_Q\dual\to\cH_Q$ of Proposition \ref{prop:sympoly} sends the product $\delta^\vee$ of the commutative associative algebra $\P_Q^\vee$ to the shuffle product of $\cH_Q$. 
\end{theorem}

\begin{proof}
Suppose that $\zeta\in\P_{Q,\bd}^\vee$, $\xi\in\P_{Q,\be}^\vee$
and $\lb i=(i_1,\dots,i_p)$ is a sequence of vertices such that $\sum_{k}e_{i_k}=\bd+\be$.
According to the proof of Proposition \ref{prop:sympoly}, 
\begin{equation}\label{eq:matrixelement}
F_{\delta^\vee(\zeta\otimes\xi)}=
\prod_{k<l}(z_k-z_l)^{\cN(i_k,i_l)}\cdot
\ang{\delta^\vee(\zeta\otimes\xi), Y(i_1,z_1)\dots Y(i_p,z_p)\one}.
\end{equation}  
The map $\delta^\vee$ is completely defined by  
$$\ang{\delta^\vee(\zeta\otimes\xi), u}=\ang{\zeta\otimes\xi, \delta(u)}$$
for all $u\in \P_Q$.
Note that 
$$\delta(Y(i_1,z_1)Y(i_2,z_2)\dots Y(i_p,z_p)\one)
=\sum_{A\sqcup B=[1,p]} (-1)^{\kappa(A,B)}\cdot
\prod_{a\in A}Y(i_{a},z_{a})\one \otimes 
\prod_{b\in B} Y(i_{b},z_{b})\one ,$$
where $\kappa(A,B)=\sum_{a>b}\chi({i_a},{i_b})$, for $A\sqcup B=[1,p]=\set{1,\dots,p}$, and the products over $A$ and $B$ are taken in the increasing order. 
For the calculation of the symmetric polynomial \eqref{eq:matrixelement}, we need to consider only $A,B$ with $\udim(A)=\sum_{a\in A}e_{i_a}=\bd$ and $\udim(B)=\be$. Recalling that the pairing of $\P_{Q,\bd}^\vee\otimes \P_{Q,\be}^\vee$ with $\P_{Q,\bd}\otimes \P_{Q,\be}$ produces an extra sign $(-1)^{\hi(\bd,\be)}$ according to the braiding, we see that $F_{\delta^\vee(\zeta\otimes\xi)}$ is equal to
 \[
\prod_{k<l}(z_k-z_l)^{\cN(i_k,i_l)}\cdot
\sum_{A\sqcup  B=[1,p]}
(-1)^{\kappa(A,B)+\hi(\bd,\be)}
\ang{\zeta,\prod_{a\in A} Y(i_{a},z_{a})\one}
\ang{\xi,\prod_{b\in B} Y(i_{b},z_{b})\one}.
 \]
The product $\prod_{k<l}(z_k-z_l)^{\cN(i_k,i_l)}$ is equal to the product of four terms, according to whether each of the two elements $k$ and $l$ belongs to $A$ or to $B$. 
Note that  
 \[
\prod_{a<a'}(z_{a}-z_{a'})^{\cN(i_{a},i_{a'})} 
\ang{\zeta,\prod_{a\in A} Y(i_{a},z_{a})\one}
 \]
is precisely $F_\zeta$, and 
 \[
\prod_{b<b'}(z_{b}-z_{b'})^{\cN(i_{b},i_{b'})} 
\ang{\xi,\prod_{b\in B} Y(i_{b},z_{b})\one}
 \]
is precisely  $F_\xi$, so we just need to investigate the factor by which their product is multiplied, that is,
\begin{multline*}
(-1)^{\kappa(A,B)+\hi(\bd,\be)}
\prod_{\ov{a<b}{a\in A,b\in B}} (z_{a}-z_{b})^{\cN(i_{a},i_{b})}
\prod_{\ov{b<a}{a\in A,b\in B}} (z_{b}-z_{a})^{\cN(i_{b},i_{a})} 
\\
=(-1)^{\kappa(A,B)+\hi(\bd,\be)}
\prod_{\ov{b<a}{a\in A,b\in B}}(-1)^{\cN(i_{a},i_{b})}
\prod_{a\in A, b\in B} (z_{a}-z_{b})^{\cN(i_a,i_b)}, 
\\
=(-1)^{\hi(\bd,\be)}
\prod_{a\in A, b\in B}(z_{a}-z_{b})^{\cN(i_a,i_b)}
=\prod_{a\in A, b\in B} (z_{b}-z_{a})^{\cN(i_a,i_b)}. 
\end{multline*}
Recalling the shuffle product formula \eqref{eq:shuffleprod}, we see that 
 \[
F_{\delta^\vee(\zeta\otimes\xi)} = F_\zeta * F_\xi,
 \]
so the isomorphism $F$ recovers precisely the shuffle product of CoHA.
\end{proof}


\subsection{A new proof of positivity of Donaldson--Thomas invariants}\label{efimov thm}
The above results together with the isomorphism
of graded vector spaces \eqref{eq:P-L} 
$$\cP_Q\iso U(\fL)\ts_{U(\fL_+)}\bQ\iso U(\fL_-)$$
have an interesting consequence: another proof of Efimov's positivity theorem for refined Donaldson--Thomas invariants \cite[Theorem 1.1]{MR2956038}. To see that, we shall study the Lie algebra $\fL_-$ in more detail.
Recall from \S\ref{sec:three} that $\fL_-$ is stable under the derivation $\dd\colon \fL\to\fL$.

\begin{theorem}
\label{th:DT}
Let $Q$ be a symmetric quiver.
Then the coalgebra $\cH_Q\dual$ has a canonical structure of a cocommutative connected vertex bialgebra.
The space of primitive elements 
$$C=P(\cH_Q\dual)\in\Vect^{L\xx\bZ}$$
is a vertex Lie algebra (also having a structure of a Lie algebra)
such that $\cH_Q\dual\iso\cU(C)$ as vertex bialgebras.
The derivation $\dd$ on $C$ has $L$-degree zero and cohomological degree $-2$,
and $C$ is a free $\bQ[\dd]$-module 
such that
the space of generators 
$$C/\dd C=\bop_{\bd\in L}W_\bd=\bop_{\bd\in L,k\in\bZ}W_\bd^k$$ 
has finite-dimensional components $W_\bd$
and $k\equiv\hi(\bd,\bd)\pmod 2$ whenever $W_\bd^k\ne0$. 
\end{theorem}



\begin{proof}
We have $\cH_Q\dual\iso\cP_Q\iso\cU(C)$ for the vertex Lie algebra $C$, see \S\ref{sec:three}.
Therefore $\cH_Q\dual$ has a structure of a cocommutative connected vertex bialgebra and $P(\cH_Q\dual)\iso C$ by Propositions \ref{li-coprod} and \ref{li-coprod2}. 
We can identify $C$ with $\fL_-$ and
we will show that $\fL_-$ satisfies the required properties.
Theorem \ref{th:VOA-CoHA} implies that we have an isomorphism of coalgebras
$$\cH_Q\dual\iso\P_Q\iso U(\fL_-).$$

Let us define another derivation of $\fL$, which we shall denote $t$; it acts on the generators of $\fL$ by the formula $t(i(n))=-i(n+1)$. By direct inspection, $t$ preserves all relations of $\fL$, and thus acts on this algebra.
Moreover, this derivation preserves the subalgebra $\fL_+$, and so acts on the vector space $\fL/\fL_+\iso \fL_-$.
We note that $[\dd,t]=\Id$ on the space of generators; it follows that $[\dd,t]=\n\bd\cdot\Id$ on each graded component $(\fL_-)_\bd$.
Thus, each such component is a $\bZ$-graded module over the Weyl algebra $A_1=\bQ[t,\dd]$ of polynomial differential operators on the line (on which $\dd$ acts by an endomorphism of weight $1$ and
$t$ acts by an endomorphism of weight $-1$).
Since we can embed $(\fL_-)_\bd\sbs\P_{0,\bd}\sbs V_\bd$ (the Fock space), weights of $(\fL_-)_\bd$ are bounded from below. Therefore by Lemma \ref{weyl rep} we obtain $\fL_-\iso W\ts\bQ[\dd]$, for $W\iso \fL_-/\dd\fL_-$.

The Fock space $V_\bd$ has finite-dimensional weight components and its weights are bounded below.
Moreover, each vector of $V_\bd$ is obtained from $\ket\bd$ by action of elements of integer weights, hence all weights of $V_\bd$ are of the form
$\oh\hi(\bd,\bd)+n$, for some $n\in\bZ$.
The corresponding degree is 
congruent to
$\hi(\bd,\bd)\pmod2$.
The same applies to $W_\bd\sbs V_\bd$ and we conclude that $\ch(W_\bd)\in\bN\lser{q^\oh}$.
It remains to establish that $W_\bd$ is finite-dimensional
and for this we will show that $\ch(W_\bd)\in\bN[q^{\pm\oh}]$.

Since $\cH_Q\dual\iso U(\cL_-)$ and $\fL_-\iso W\ts\bQ[\dd]$, we obtain
$$Z(\cH_Q,x,q\inv)=\Exp\rbr{Z(\fL_-,x,q)}
=\Exp\rbr{\frac{Z(W,x,q)}{1-q}}.
$$
Applying formula \eqref{DT1} for DT invariants, we obtain
$$Z(W,x,q)= \sum_\bd(-1)^{\hi(\bd,\bd)}\Om_\bd(q\inv)x^\bd$$
meaning that $\ch(W_\bd)=\Om_\bd(q\inv)$. Let us now use characters of \COHA-modules. 
According to \eqref{DT2},
for any $\bw\in\bN^I$, we have
$$Z(\M_{\bw},x,q\inv)=
\Exp\rbr{\sum_\bd\frac{q^{\bw\cdot \bd}-1}{q-1}(-1)^{\hi(\bd,\bd)}\Om_\bd(q\inv)x^\bd}.
$$
All components $\M_{\bw,\bd}$ of the \COHA module $\M_\bw$ are finite-dimensional (as the cohomology of an algebraic variety), 
hence $Z(\M_{\bw},x,q\inv)\in\bZ[q^{\pm\oh}]\pser{x_i\col i\in I}$.
This implies that for all $\bd$
$$
\frac{q^{\bw\cdot \bd}-1}{q-1}\ch(W_\bd)
=\frac{q^{\bw\cdot \bd}-1}{q-1}\Om_\bd(q\inv)$$
is an element of $\bQ[q^{\pm\oh}]$ (actually $\bZ[q^{\pm\oh}]$, but we don't need this). 
We can choose $\bw\in\bN^I$ such that $\bw\cdot\bd>0$. We have just shown that the product of the polynomial $\frac{q^{\bw\cdot \bd}-1}{q-1}\in \bN[q]$ and the series $\ch(W_\bd)\in\bN\lser{q^{\oh}}$ 
is a Laurent polynomial. This implies that $\ch(W_\bd)\in\bN[q^{\pm\oh}]$, as required.
\end{proof}

\begin{lemma}\label{weyl rep}
Let $A_1=\bQ[t,\dd]$ be the Weyl algebra (with $\dd t-t\dd =1$) equipped with the weight grading $\wt(\dd)=1$, $\wt(t)=-1$, and let $M=\bop_{n\in\bZ}M_n$ be a bounded below graded $A_1$-module.
Then $M$ is a free module over $\bQ[\dd]$.
More precisely, $\dd\colon M\to M$ is injective and we have $M=\bop_{n\ge0}\dd^n(V)$, for any graded subspace $V\sbs M$ such that $M=V\oplus\im(\dd)$.
\end{lemma}
\begin{proof}
Assume that $\dd v=0$ for some (homogeneous) $v\in M$.
Then the subspace $\angs{t^nv}{n\ge0}$ is an $A_1$- submodule.
It is finite-dimensional as $t^nv=0$, for $n\gg0$, by degree reasons.
But there are no nonzero finite-dimensional $A_1$-modules, hence we conclude that $v=0$.

Let $V\sbs M$ be a graded subspace such that $M=V\oplus\im(\dd)$.
If the sum $\sum_{n\ge0}\dd^n(V)$ is not direct, we can find $v_k\in V$ such that $\sum_{k=0}^n\dd^k(v_k)=0$ and $v_n\ne 0$.
As $\dd$ is injective, we can assume that $v_0\ne0$.
But $v_0\in V\cap\im(\dd)=0$, which is a contradiction.
To show that $x\in M_k$ is contained in $\sum_{n\ge0}\dd^n(V)$, we proceed by induction on $k$.
We can decompose $x=v+\dd(y)$, where $v\in V_k$ and $y\in M_{k-1}$.
By induction $y\in \sum_{n\ge0}\dd^n(V)$, hence $\dd(y)\in \sum_{n\ge1}\dd^n(V)$.
We conclude that $x=v+\dd(y)\in \sum_{n\ge0}\dd^n(V)$.
\end{proof}


\begin{remark}
In the proof of \cite[Theorem 1.1]{MR2956038}, the free action of the polynomial ring in one variable on $\cH_{Q,\bd}$ that is used to find the space of free generators of $\cH_Q$ is implemented using the multiplication by 
 \[
\sigma_\bd=\sum_{j\in I, 1\le k\le \bd_j} x_{j,k}.
 \]
This can also be obtained as a byproduct of our argument.
Recalling the definition of the endomorphism $t$ from the proof of Theorem \ref{th:DT}, we see that this multiplication is precisely the action of the endomorphism $t\dual$ on $\P_Q\dual$.
In that theorem, we saw that $\P_Q$ is isomorphic to $U(\fL_-)\iso S^c(W\otimes\bQ[\dd])$ as a coalgebra; after taking graded duals (which can be implemented by doing the Fourier transform for differential operators), the modules become $\bQ[t\dual]$-free. 
\end{remark}

Let us record a simple consequence of Theorem \ref{th:DT}. 

\begin{corollary}\label{cor:DT}
For any symmetric quiver $Q$, consider the graded vertex bialgebra $\cH_Q\dual$ and its space of primitive elements $C=P(\cH_Q\dual)$, which is a vertex Lie algebra.
Then the corresponding Donaldson--Thomas invariants satisfy
$$\Om_\bd(q\inv)=\ch(C_\bd/\dd C_\bd)\in\bN[q^{\pm\oh}].$$
\end{corollary}
\begin{proof}
We established that $\cH_Q\dual\iso\P_Q$ and that the principal free vertex algebra $\cP_Q$ is isomorphic to universal enveloping vertex algebra of a \LieCon algebra $C$.
We have seen in the proof of Theorem ~\ref{th:DT} that $\Om_\bd(q\inv)=\ch(W_\bd)$,
where $\fL_-=W\ts\bQ[\dd]$, so that $W\iso\fL_-/\dd\fL_-$.
Recall that for any \LieCon algebra $C$ its coefficient algebra $\fL$ has a decomposition $\fL=\fL_-\oplus\fL_+$ such that $C\iso\fL_-$ (as a graded vector space).
Moreover, we have
\begin{equation}
\fL/\dd\fL\iso\fL_-/\dd\fL_-\iso C/\dd C,
\end{equation}
which completes the proof.
\end{proof}

This statement is easily generalisable to the following appealing conjecture suggesting a relationship between vertex algebras and more general DT invariants.

\begin{conjecture}
For each \coha \cH associated to a (symmetric) quiver with potential, its dual $\cH\dual$ can be equipped with a vertex bialgebra algebra structure
(\cf \cite{joyce_ringel})
such that the corresponding perverse graded object is
isomorphic to the universal enveloping vertex algebra of some \LieCon algebra $C$.
The corresponding DT invariants are equal to the characters of the components of $C/\dd C$.
\end{conjecture}

\defcite\joy{joyce_ringel}

\subsection{Relationship to the vertex algebras of Joyce}
In this section we will briefly explain the relationship of our results to the geometric construction of vertex algebras proposed in \joy (see also \cite{gross_homology,gross_universal,bojko_wall,latyntsev_cohomological}).
In order to do this, we will need to formulate a minor generalization of that construction.
Let us assume that we have the following data
\begin{enumerate}
\item 
A lattice $L$ equipped with a symmetric bilinear form $\hi$.

\item An abelian category $\cA$
and a linear map $\cl:K_0(\cA)\to L$.

\item A moduli stack $\cM$ of object in $\cA$ such that the substack $\cM_\bd$ of objects $E\in\cM$ with $\cl(E)=\bd$ is open and closed and there exist natural morphisms of stacks
\begin{enumerate}
\item
$\Phi:\cM\xx\cM\to\cM$ that maps $(E,F)\mto E\oplus F$.

\item 
$\Psi:B\Gm\xx\cM\to\cM$ that maps $\Gm\xx\Aut (E)\ni(t,f)\mto tf\in\Aut(E)$, for $E\in\cA$.
\end{enumerate}

\item
A perfect complex $\Te$ on $\cM\xx\cM$ such that
\begin{enumerate}
\item
\Te is weakly-symmetric, meaning that
$$[\si^*\Te\dual]=[\Te]$$
in the Grothendieck group of $\cM\xx\cM$,
where $\si:\cM\xx\cM\to\cM\xx\cM$ is the permutation of factors.
\item The restriction $\Te_{\bd,\be}=\Te|_{\cM_\bd\xx\cM_\be}$ has constant rank $\hi(\bd,\be)$ for all $\bd,\be\in L$.
\item
We have
$$(\Phi\xx\id_\cM)^*\Te\iso
\pi_{13}^*\Te\oplus\pi_{23}^*\Te,\qquad
(\id_\cM\xx\Phi)^*\Te\iso\pi_{12}^*\Te\oplus\pi_{13}^*\Te,
$$
$$(\Psi\xx\id_\cM)^*\Te\iso\cU\boxtimes\Te,\qquad
(\id_\cM\xx\Psi)^*\Te\iso\cU\dual\boxtimes\Te,
$$
where $\pi_{ij}:\cM^3\to\cM^2$ is the projection to the corresponding factors and 
$\cU$ is the universal line bundle over $B\Gm$.
\end{enumerate}
\end{enumerate}

The proof of the following result goes through the same lines as in \joy.

\begin{theorem}\label{th:joyce}
Consider the $L\xx\bZ$-graded vector space
$$V=\bop_{\bd\in L} V_\bd
=\bop_{\bd\in L} H_*(\cM_\bd)[\hi(\bd,\bd)]$$
so that the component of $V_\bd$ of homological degree $k$ (and weight $\oh k$) is $H_{k-\hi(\bd,\bd)}(\cM_\bd)$.
Then $V$ has a structure of a graded vertex algebra (in the symmetric monoidal category $\Vect^{L\xx\bZ}$ with the braiding induced by $\hi$) defined by
\begin{enumerate}
\item $\one=\eta_*(1)\in H_0(\cM_0)$, where $\eta:\pt\to\cM$
is the inclusion of the zero object.
\item
The operator $T:V\to V$ of homological degree $2$ (and weight $1$)
is defined by $T(v)=\Psi_*(t\boxtimes v)$, where $t\in H_2(B\Gm)=H_2(\bP^\infty)$ is the canonical generator. 

\item For $u\in H_k(\cM_\bd)$, $v\in H_*(\cM_\be)$, we define
$$Y(u,z)v=
(-1)^{k\hi(\bd,\bd)}z^{\hi(\bd,\be)}
\Phi_*\rbr{(e^{zT}\otimes\id)\rbr{(u\boxtimes v)\cap c_{z\inv}(\Te_{\bd,\be})}}.$$
\end{enumerate}
\end{theorem}

As before, let $Q$ be a symmetric quiver with the set of vertices $I$.
Let $L=\bZ^I$ and $\hi$ be the Euler form of $Q$.
Let $\cA$ be the category of representations of $Q$ and $\cM$ be the stack of representations of $Q$.
We have a linear map $\udim:K_0(\cA)\to L$.
We define the perfect complex $\Te=\RHom$ over $\cM\xx\cM$ such that its fiber over $(M,N)\in\cM\xx\cM$ is isomorphic to $\RHom(M,N)\in D^b(\Vect)$.
The rank of $\Te$ over $\cM_\bd\xx\cM_\be$ is equal to $\hi(\bd,\be)$.

\begin{lemma}
We have
$$[\si^*\Te\dual]=[\Te]$$
in the Grothendieck group of $\cM\xx\cM$.
\end{lemma}
\begin{proof}
Let $A=\bC Q$ be the path algebra of $Q$.
For every $i\in I$, let $e_i\in A$ be the corresponding idempotent and $P=Ae_i$ be the corresponding projective $A$-module.
For any representation $M$, we have the standard projective resolution
$$0\to\bop_{a:i\to j}P_j\ts M_i\to\bop_{i}P_i\ts M_i\to M\to0.$$
Therefore $\RHom(M,N)$ can be written as a complex
$$\dots\to0\to\bop_{i}\Hom(M_i,N_i)\to\bop_{a:i\to j}\Hom(M_i,N_j)\to0\to\dots$$
Similarly, $\RHom(N,M)$ can be written as a complex
$$\dots\to0\to\bop_{i}\Hom(N_i,M_i)\to\bop_{a:i\to j}\Hom(N_i,M_j)\to0\to\dots$$
Using the fact that $Q$ is symmetric and that $\Hom(N_i,M_j)\iso\Hom(M_j,N_i)\dual$,
we can rewrite this complex in the form
$$\dots\to0\to\bop_{i}\Hom(M_i,N_i)\dual\to
\bop_{a:i\to j}\Hom(M_i,N_j)\dual\to0\to\dots$$
This implies that 
$[\RHom(M,N)\dual]=[\RHom(N,M)]$.
The statement of the lemma is a global version of this observation.
\end{proof}

Applying the previous lemma and the construction of Theorem 
\ref{th:joyce} we obtain a vertex algebra structure on $\cH_Q\dual$ (\cf Proposition \ref{prop:sympoly} where we proved that $\cH_Q\dual\iso\cP_Q$ as vector spaces).
Note that in the original version of Theorem \ref{th:joyce} proved in \joy, one requires that $\Te$ is symmetric, meaning that $\si^*\Te\dual\iso\Te\dual[2n]$ for some $n\in\bZ$.
Despite of the previous lemma, this condition is generally not satisfied by $\Te=\RHom$ for symmetric quivers (otherwise the category of quiver representations would be Calabi-Yau).
Because of this technical difficulty, one considered in \joy only vertex algebras associated to $\Te=\RHom$ for Calabi-Yau categories or
to the symmetrized perfect complex
$\Te=\RHom\oplus\si^*(\RHom\dual)$ for the category of quiver representations.
The latter choice leads to vertex algebras not directly related to \cohas.

On the other hand, in \cite{latyntsev_cohomological} one associated quantum vertex algebras to \Te which are not necessarily symmetric.
Because of the previous lemma, in the case of symmetric quivers and $\Te=\RHom$, the resulting quantum vertex algebra is actually a vertex algebra (in an appropriate symmetric monoidal category).
Theorem \ref{th:DT} which states, in particular,
that $\cH_Q\dual$ has a structure of a vertex bialgebra should be compared to the result of \cite{latyntsev_cohomological} about the quantum vertex bialgebra structure on $\cH_Q\dual$.

\section{CoHA-modules and free vertex algebras}\label{sec:iso-modules}

In this section, we shall use previously obtained identification $\cH_Q\dual\cong\P_Q$ to give a new interpretation of the \COHA modules 
 \[
\cM_\bw=\cH_Q/\rbr{e_\bd^\bw\cH_{Q,\bd}\col \bd>0}.
 \]
considered in Section \ref{sec:modules}. Since $\cM_\bw$ is a quotient of $\cH_Q$, it is natural to seek for a description of $\cM_\bw\dual$ as a subspace of $\P_Q$. In this section, we give two such descriptions. In \S\ref{modif coproduct}, we interpret $\cM_\bw\dual$ in terms of the coproduct generalising the coproduct $\P_Q\to\P_Q\otimes\P_Q$ considered earlier. In \S\ref{subspace}, we exhibit a combinatorially defined spanning set of $\cM_\bw\dual\subset\P_Q$.  

\subsection{Modified coproduct}\label{modif coproduct}

For each $\bw\in\bN^I$, let us consider the subset
$$X_{\bw}=\sets{i(n)\in X}{n\ge-\bw_i}$$
of the set of generators of the Lie algebra $\fL$, and the corresponding $U(\fL)$-module $\P_\bw=U(\fL)/U(\fL)X_\bw$ with the cyclic vector $\one_\bw$.
For each $i\in I$, the series $\mathbf{i}(z)=\sum_{n\in\bZ}i(n)z^{-n-1}$ defines a field on $\P_\bw$.
These series can be used to equip $\P_\bw$ with a structure of a module over the vertex algebra $\P_Q$; clearly, $\P_Q$ itself is a particular case of this construction for $\bw=0$. 

Note that since $\bw\in\bN^I$, we have $X_+\subset X_\bw$, and therefore there is a canonical surjection of $U(\fL)$-modules $\pi\colon\P_Q\to \P_\bw$. We shall now see how the graded dual 
$$
\P_\bw^\vee=\bigoplus_{\bd\in\bN^{I}}\P_{\bw,\bd}\dual
$$
 of $\P_\bw$ is included in $\P_Q\dual$. As before, we fix a dimension vector $\bd\in\bN^I$ and let $p=\n\bd$ and $\lb j=(j_1,\dots,j_p)$ be any sequence of vertices with $e_{j_1}+\cdots+e_{j_p}=\bd$.

\begin{lemma}\label{lm:sympoly1}
For any $\ksi\in\cP_{\bw,\bd}\dual$, the Laurent series 
 $$ 
F_{\ksi,\lb j}(z_1,\ldots,z_p)=\prod_{k<l}(z_k-z_l)^{\cN(j_k,j_l)}\cdot
\ang{\xi, Y(j_1,z_1)Y(j_2,z_2)\cdots Y(j_p,z_p)\one_\bw}.
 $$
is completely symmetric under the action of $\Si_p$ permuting simultaneously the vertices $j_p$ and the variables $z_p$. Moreover, it is a polynomial divisible by the product $z_1^{\bw_{j_1}}\cdots z_p^{\bw_{j_p}}$.
\end{lemma}

\begin{proof}
The proof is analogous to that of Lemma \ref{lm:sympoly}; the only difference is that $X_+$ is replaced by $X_\bw$, which has the effect of replacing the property of absence of negative powers by the property of absence of powers of $z_p$ that are less than $\bw_{j_p}$. 
\end{proof}

\begin{proposition}\label{prop:sympoly1}
For any $\bd\in\bN^I$, 
we have an isomorphism of graded vector spaces 
 $$
F\colon \P_{\bw,\bd}\dual\to e_{\bd}^\bw\La_\bd[-\hi(\bd,\bd)],\qquad
e_{\bd}^\bw
=\prod_{i\in I}\prod_{k=1}^{\bd_i}x_{i,k}^{\bw_i},
$$ 
defined by the formula
$$
\ksi\mapsto F_\ksi=\prod_{k<l}(z_k-z_l)^{\cN(i_k,i_l)}\cdot
\ang{\xi, Y(i_1,z_1)Y(i_2,z_2)\cdots Y(i_p,z_p)\one_\bw},
$$
where $p=\n\bd$, $i_{\bd_1+\dots+\bd_{j-1}+k}=j$ and $z_{\bd_1+\dots+\bd_{j-1}+k}=x_{j,k}$,
for $j\in I$ and $1\le k\le \bd_j$.
\end{proposition}

\begin{proof}
Completely analogous to that of Proposition \ref{prop:sympoly}. 
\end{proof}

We note that $e_{\bd}^\bw\La_\bd[-\hi(\bd,\bd)]\iso e_{\bd}^\bw\cH_{Q,\bd}$ is precisely one of the vector spaces used in the shuffle algebra description \eqref{coha module1} of the module $\cM_\bw$.  
To use this observation, we consider the surjection $\pi:\P_Q\to\P_\bw$ and define a $U(\fL)$-module map 
 $$
\rho_\bw:=(\Id\otimes\pi)\delta\colon\P_Q\to \P_Q\otimes\P_\bw$$
which can be interpreted as a coaction of the cocommutative coalgebra $\P_\bw$ on its comodule $\P_Q$.
Explicitly, we have
\begin{equation}\label{explicit coprod}
\rho_\bw(i_1(n_1)i_2(n_2)\dots i_p(n_p)\one)
=\sum_{A\sqcup B=\set{1,\dots,p}} (-1)^{\kappa(A,B)}\cdot
\prod_{a\in A}i_{a}(n_{a})\one \otimes 
\prod_{b\in B} i_{b}(n_{b})\one_{\bw},
\end{equation}
where $\kappa(A,B)=\sum_{a>b}\chi({i_a},{i_b})$.
Note that the projection of $\rho_\bw(v)$ to $\P_Q\ts\P_{\bw,0}$ is equal to $v\ts \one_\bw$.
Since the description of the module $\cM_\bw$ uses the spaces $e_{\bd}^\bw\cH_{Q,\bd}$ with $\bd>0$, it will be useful to consider the reduced coaction
$$\bar\rho_\bw:\P_Q\xto{\rho_\bw} \P_Q\otimes\P_\bw\to
\P_Q\otimes\bar\P_\bw,\qquad
\bar\P_\bw=\P_\bw/\P_{\bw,0}.
$$
Using the isomorphism $\bar\P_\bw\iso\bop_{\bd>0}\P_{\bw,\bd}$, we can write $\bar\rho_\bw$ in the form
$$\bar\rho_\bw(v)=\rho_\bw(v)-v\otimes\one_\bw.$$

\begin{theorem}\label{th:modkernel}
The kernel of the map $\bar{\rho}_\bw$ is isomorphic to the graded dual of the \COHA-module 
$$\cM_\bw=\cH_Q/\rbr{e_\bd^\bw\cH_{Q,\bd}\col \bd>0}.$$
\end{theorem}

\begin{proof}
Note that we have an exact sequence 
 \[
0\to\ker(\bar{\rho}_\bw)\hookrightarrow\P_Q\xto{\bar\rho_\bw}_\bw  \P_Q\otimes\bar\P_\bw,
 \]
which, after passing to graded duals, becomes 
 \[
0\leftarrow\ker(\bar{\rho}_\bw)^\vee\twoheadleftarrow\P_Q^\vee\xlto{\bar\rho_\bw\dual}\P_Q^\vee\otimes\bar\P_\bw^\vee.
 \]
Using the isomorphisms of Propositions \ref{prop:sympoly}, \ref{prop:sympoly1}, we obtain the diagram
\begin{ctikzcd}
0&\lar\ker(\bar\rho_\bw)\dual\dar&\P_Q\dual\lar\dar["F"]
&\P_Q\dual\ts\bar\P_\bw\dual
\dar["F\ts F"]\lar["\bar\rho_\bw\dual"']\\
0&\lar \M_\bw&\cH_Q\lar&\cH\ts\Big(\bop_{\bd>0}e_\bd^\bw\cH_{Q,\bd}\Big)\lar["*"']
\end{ctikzcd}
where the bottom right map is the shuffle product by
Theorem \ref{th:VOA-CoHA}. Its cokernel is isomorphic to
the \COHA module $\cM_\bw$ by \eqref{coha module1}.
Therefore we obtain an isomorphism
$\ker(\bar{\rho}_\bw)^\vee\iso\cM_\bw$.
\end{proof}

\subsection{Subspace construction}\label{subspace}

We shall now exhibit an explicit combinatorial spanning set of the subspace $\ker(\bar{\rho})\subset\P_Q$.
Let $\Q_\bw$ be the
subspace of $\P_Q\iso U(\fL)\ts_{U(\fL_+)}\bQ$
obtained by applying elements of $X_\bw=\sets{i(n)}{n\ge-\bw_i}$ to the vacuum $\one\in\P_Q$.

We can also interpret this space as follows.
Let $\fL^\bw_+$ denote the Lie subalgebra of $\fL$ generated by the set $X_\bw$.
As $X_+\sbs X_\bw$, for $\bw\in\bN^I$, 
we have $\fL_+\subset \fL^\bw_+$.
Then $\Q_\bw$ can be identified with
$$\Q_\bw=U(\fL^\bw_+)\ts_{U(\fL_+)}\bQ,$$
where $\bQ$ is equipped with the structure of the trivial $\fL_+$-module.

We can actually identify $\Q_\bw$ with the universal enveloping algebra of a certain Lie algebra as follows.
Consider the Lie algebra isomorphism
\begin{equation}
\tau_\bw\colon \fL\to\fL,\qquad i(n)\mto i(n-\bw_i),
\end{equation}
which extends to the isomorphism $U(\fL)\to U(\fL)$.
Note that $\ta_\bw$ maps $\fL_+$ to $\fL^\bw_+$ isomorphically; 
we also define a new Lie algebra $\fL^\bw_-=\tau_\bw(\fL_-)$, 
leading to the direct sum decomposition $\fL=\fL^\bw_-\oplus\fL^\bw_+$. 
As $\fL_+\subset \fL^\bw_+$,
we have direct sum decompositions
\begin{equation}\label{decomp1}
\fL=\fL^\bw_-\oplus\fLm \oplus \fL_+,\qquad
\fL^\bw_+=\fLm \oplus \fL_+,\qquad
\fLm=\fL_-\cap \fL^\bw_+,
\end{equation}
where $\fLm$ is a Lie algebra, being an intersection of two Lie algebras.
We conclude that
\begin{equation}\label{eq:dualmodule-smallliealg}
\Q_\bw=U(\fL^\bw_+)\ts_{U(\fL_+)}\bQ\iso U(\fLm ).
\end{equation}

\begin{theorem}\label{module spanning}
For each $\bw\in\bN^I$, we have a commutative diagram
\begin{ctikzcd}
\P_Q\dual\rar["F"]\dar[->>]&\cH_Q\dar[->>]\\
\Q_\bw\dual\rar["\sim"]&\M_\bw
\end{ctikzcd}
\end{theorem}

\begin{proof}
According to Theorem \ref{th:modkernel}, we have $\M_\bw\dual\iso \ker(\bar{\rho}_\bw)$,
so we need to show that $\ker(\bar{\rho}_\bw)=\Q_\bw$.
Let us first remark that $\Q_\bw\subset\ker(\bar{\rho}_\bw)$.
Indeed, if 
$$v=i_1(n_1)i_2(n_2)\dots i_p(n_p)\one\in\Q_\bw$$ with $i_k(n_k)\in X_\bw$ for all $k$, then in Formula \eqref{explicit coprod} we have $\prod_{b\in B} i_{b}(n_{b})\one_{\bw}=0$ whenever $B\ne\es$.
Therefore $\rho(v)=v\ts\one_\bw$ and $\bar\rho(v)=0$.

To establish that the inclusion $\Q_\bw\subset\ker(\bar{\rho}_\bw)$ is an equality, it is sufficient to show that these two subspaces of $\P_Q$ have the same Poincar\'e series, that is,
 \[
Z(\Q_\bw,x,q)=Z(\ker(\bar{\rho}),x,q). 
 \]
We shall now prove this equality. The left hand side computation will use various results about \LieCon algebras. We have already seen that there is an isomorphism $\Q_\bw\iso U(\fLm)$. The decomposition \eqref{decomp1} implies, together with the Poincar\'e--Birkhoff--Witt theorem, that we have 
an isomorphism of $L\xx\bZ$-graded vector spaces 
 $$
U(\fL_-)\iso
U(\fL)\ts_{U(\fL_+)}\bQ\iso
U(\fL^\bw_-)\otimes U(\fLm), 
 $$
and therefore  
 $$
Z(U(\fL_-),x,q)=Z(U(\fL^\bw_-),x,q)\cdot Z(U(\fLm),x,q),
 $$
or, equivalently, 
 $$
Z(U(\fLm),x,q)=Z(U(\fL_-),x,q)\cdot Z(U(\fL^\bw_-),x,q)^{-1}.
 $$
According to Corollary \ref{cor:char free vertex}, we have 
 \[
Z(U(\fL_-),x,q)=Z(\cH_Q,x,q^{-1})=A_Q(x,q^{-1}). 
 \]
This formula can be used to determine the Poincar\'e series of $U(\fL^\bw_-)=\ta_\bw(U(\fL_-))$. Indeed, we have $\wt(\ta_\bw(i(n))=\wt(i(n))+\bw_i$.
Therefore, for any graded subspace $M\sbs U(\fL)_\bd$, we have $\ch(\ta_\bw(M))=q^{\bw\cdot\bd}\ch(M)$ and, for any $L\xx\bZ$-graded
subspace $M\sbs U(\fL)$, we have
$$Z(\ta_\bw(M),x,q)=S_{2\bw}Z(M,x,q),$$
where $S_{\bw}(x^\bd)=q^{\oh\bw\cdot\bd}x^\bd$. This implies
 \[
Z(U(\fL^\bw_-),x,q)
=S_{2\bw}Z(U(\fL_-),x,q)
=S_{2\bw}A_Q(x,q^{-1}),
 \]
and so we may conclude that 
 \[
Z(\Q_\bw,x,q)=Z(U(\fLm),x,q)=A_Q(x,q^{-1})\cdot S_{2\bw}A_Q(x,q^{-1})\inv. 
 \]
At the same time, according to Theorem \ref{th:modkernel}, we have
 \[
Z(\ker(\bar{\rho}),x,q)=Z(\cM_\bw,x,q^{-1}).
 \]
Since Proposition \ref{char module} asserts that $Z(\M_{\bw},x,q)=A_Q(x,q)\cdot S_{-2\bw}A_Q(x,q)\inv$, we conclude that 
 \[
Z(\ker(\bar{\rho}),x,q)=A_Q(x,q^{-1})\cdot S_{2\bw}A_Q(x,q^{-1})\inv=Z(\Q_\bw,x,q),
 \]
which completes the proof.
\end{proof}

\begin{remark}
Let us consider the quiver $Q$ with one vertex $1$ and two loops. Its Euler form on $\bZ^I\iso \bZ$ is given by $\hi(e_1,e_1)=-1$, and the Lie algebra $\fL$ is the Lie superalgebra generated by odd elements $a_n$, $n\in\bZ$, such that 
 \[
[a_m,a_n]=[a_{m-1},a_{n+1}], \qquad m,n\in\bZ.
 \]
This is precisely the algebra considered in~\cite{MR2480715}, where it was proved that for $\bw=1$ and $\bd=n$, the dimension of $\Q_{\bw,\bd}$ is given by the $n$-th Catalan number, and incorporating the weight leads to the $q$-analogue of the Catalan numbers introduced by Carlitz and Riordan~\cite{MR168490}.
We can recover this result using Theorem \ref{module spanning}, which asserts that we should check the same for the dimension of $\cM_{\bw,\bd}$. The latter vector space is, up to a degree shift, isomorphic to the cohomology of the non-commutative Hilbert scheme
$\rH^{(2)}_{n,1}=\Hilb_{n,1}$.
According to~\cite{reineke_cohomology}, its dimension is indeed the $n$-th Catalan number. Moreover, the Poincar\'e polynomial of 
$\rH^{(2)}_{n,1}$ computed in \cite{reineke_cohomology} is easily seen to produce the expected $q$-Catalan number. 
\end{remark}

Recall that $\Q_\bw$ is spanned by elements obtained by applying elements of $$X_\bw=\sets{i(n)}{n\ge-\bw_i}$$
to the vacuum $\one\in\P_Q$. Extracting from these elements a basis of $\Q_\bw$ is far from obvious; for example, it does not seem that either of the two known combinatorial descriptions of a basis in $\P_Q$ (obtained in \cite[Theorem 1]{MR1935501} and in \cite[Theorem 4.8]{MR2967107}) is easy to use to describe a basis of $\Q_\bw$. If one is in the situation for which the locality relations form a Gr\"obner basis (for the most obvious order of monomials, such quivers are classified in \cite{DoFeRe}), such a description exists. For example, for each $n\ge 1$, if one considers the quiver on $n$ vertices with two loops at each vertex and one arrow $i\to j$ for each $i\ne j$, and the framing vector $\bw=(1,1,\ldots,1)$, the locality relations form a Gr\"obner basis and one recovers the result of \cite[Theorem 1]{MR2480715} stating that the space $\Q_\bw$ has an explicit combinatorial basis labeled by parking functions on $\{1,\ldots,n\}$.

On the other hand, one can describe an explicit basis of $\M_\bw$ parametrized by subtrees of the tree of paths in the framed quiver
\cite{reineke_cohomology,engel_smooth} (this basis depends on some non-canonical choices).
In the recent paper \cite{franzen_tautological} a canonical basis of $\M_\bw$ was constructed.
Taking the dual basis of $\M_\bw\dual\iso\Q_\bw$, we obtain a canonical combinatorial basis of $\Q_\bw$.
\subsection{Poincar\'e series of the Lie algebra \texorpdfstring{$\fL_+$}{L+} }\label{sec:strange-symmetry}

We shall now use our results on the Poincar\'e series of CoHA-modules to establish a surprising symmetry formula, showing that the Poincar\'e series of Lie algebras $\fL_-$ and $\fL_+$ add up to zero. To the best of our knowledge, this result does not follow from the general principles. It would be interesting to determine the class of vertex Lie algebras for which such symmetry holds.

\begin{theorem}\label{th:symmetry}
The Poincar\'e series $Z(\fL_-,x,q)$ and
$Z(\fL_+\dual,x,q)$ belong to the subring 
 \[
 \bQ(q^\oh)\pser{x_i\col i\in I}\subset \bQ\lser{q^\oh}\pser{x_i\col i\in I}.
 \]
In that subring, we have the equality 
$$Z(\fL_-,x,q\inv)=-Z(\fL_+\dual,x,q).$$
\end{theorem}
\begin{proof}
We have $\fL=\fL^\bw_-\oplus\fLm \oplus \fL_+$
and $\fL^\bw_+=\fLm \oplus \fL_+$ by \eqref{decomp1}.
Therefore
\begin{equation*}
Z(\fL_-,x,q)=Z(\fL^\bw_-,x,q)+Z(\fLm,x,q),
\end{equation*}
\begin{equation*}
Z((\fL^\bw_+)\dual,x,q)
=Z((\fLm)\dual,x,q)+Z(\fL_+\dual,x,q).
\end{equation*}
Note that
$$Z(\fL^\bw_-,x,q)=Z(\ta_\bw\fL_-,x,q)
=S_{2\bw}Z(\fL_-,x,q),$$
$$Z((\fL^\bw_+)\dual,x,q)=Z((\ta_\bw\fL_+)\dual,x,q)
=S_{-2\bw}Z(\fL_+\dual,x,q),$$
hence
\begin{equation}\label{eq:Lm}
Z(\fLm,x,q)=Z(\fL_-,x,q)-S_{2\bw}Z(\fL_-,x,q),
\end{equation}
\begin{equation}\label{eq:Lp}
Z((\fLm)\dual,x,q)
=S_{-2\bw}Z(\fL_+\dual,x,q)-Z(\fL_+\dual,x,q)
\end{equation}
By Corollary \ref{cor:char free vertex},
we have
$$\Exp(Z(\fL_-,x,q))=Z(U(\fL_-),x,q)=
\sum_{\bd\in\bN^I}
\frac{(-q^{\oh})^{\hi(\bd,\bd)}}{(q)_\bd}x^\bd,$$
hence this series as well as $Z(\fL_-,x,q)$ are contained in  $\bQ(q^\oh)\pser{x_i\col i\in I}$ and we can interpret \eqref{eq:Lm} as an equality in this ring and obtain
\begin{equation}\label{eq:Lm1}
Z(\fLm,x,q\inv)=Z(\fL_-,x,q\inv)-S_{-2\bw}Z(\fL_-,x,q\inv),
\end{equation}

We have $U(\fLm)\iso \M_\bw\dual$ by Theorem \ref{module spanning} and \eqref{eq:dualmodule-smallliealg}.
All components of the module $\M_\bw$ are (degree shifted) cohomology of algebraic varieties, so they are finite-dimensional.
Therefore components of $\fLm$ are also finite-dimensional,
meaning that the coefficients of the power series $Z(\fLm,x,q)$ are Laurent polynomials in $q^{\oh}$ and
$Z(\fLm,x,q\inv)=Z((\fLm)\dual,x,q)$.
Comparing \eqref{eq:Lm1} and \eqref{eq:Lp}, we obtain
$$S_{2\bw}Z(\fL_-,x,q\inv)-Z(\fL_-,x,q\inv)
=Z((\fL^\bw_+)\dual,x,q)
-S_{2\bw}Z(\fL_+\dual,x,q).
$$
Letting $\bw\to\infty$ (meaning that $\bw_i\to\infty$ for all $i\in I$)
we obtain the required equation.
\end{proof}

This result implies that there exists a version of Corollary \ref{cor:DT} where the action of $\partial$ on the Lie algebra $\fL_+$ is used to determine the refined Donaldson--Thomas invariants.
As a consequence, one obtains a strong supporting evidence for the Koszulness conjecture of~\cite{DoFeRe}.
More precisely, 
in \cite{DoFeRe} one constructed some explicit quadratic algebra $\cA_Q$ such that its Koszul dual algebra is isomorphic to $U(\fL_+)$ and
\begin{equation}
Z(\cA_Q,q^\oh x,q)=A_Q(x,q).
\end{equation}
By Theorem \ref{th:symmetry}, we have
\begin{equation}
Z(U(\fL_+)\dual,x,q)=Z(U(\fL_-),x,q\inv)\inv=A_Q(x,q)\inv,
\end{equation}
hence
\begin{equation}
Z(U(\fL_+)\dual,x,q)\cdot Z(\cA_Q,q^\oh x,q)=1
\end{equation}
which is the ``numerical Koszulness'' property of the algebra $\cA_Q$.

\bibliography{newbib}
\bibliographystyle{hamsplain}
\end{document}